\documentclass[12pt]{amsart}
\usepackage{amssymb, amscd, amsmath, amsthm, latexsym, enumerate}
\usepackage[all]{xy}
\usepackage{mathrsfs}

\usepackage{mathtools}

\date{May 22, 2023}

\newtheorem{dummy}{anything}[section]
\newtheorem{theorem}[dummy]{Theorem}
\newtheorem*{thma}{Theorem A}
\newtheorem*{thmb}{Theorem B}

\newtheorem{lemma}[dummy]{Lemma}
\newtheorem{proposition}[dummy]{Proposition}
\newtheorem{corollary}[dummy]{Corollary}

\theoremstyle{definition}%%Change Theoremstyle
\newtheorem{definition}[dummy]{Definition}
  \newtheorem{example}[dummy]{Example}
  
  \newtheorem{remark}[dummy]{Remark}

  \newtheorem*{acknowledgement}{Acknowledgement}

\newcommand
{\eqncount}{\setcounter{equation}{\value{dummy}}%
\addtocounter{dummy}{1}}

\usepackage{color}

\newcommand{\bZ}{\mathbb Z}
\newcommand{\bM}{\mathbb M}
\newcommand{\bL}{\mathbb L}
\newcommand{\cy}[1]{\bZ/{#1}}
\DeclareMathOperator{\Isom}{Isom}
\DeclareMathOperator{\Self}{Self}
\DeclareMathOperator{\Homeo}{Homeo}
\DeclareMathOperator{\Aut}{Aut}
\DeclareMathOperator{\Out}{Out}
\newcommand{\wB}{\widetilde B}
\newcommand{\wM}{\widetilde M}
\newcommand{\wN}{\widetilde N}
\newcommand{\wX}{\widetilde X}

\newcommand{\cS}{{S}_{TOP}}

\newcommand{\cF}{\mathscr F}

\newcommand{\la}{\langle}
\newcommand{\ra}{\rangle}
\newcommand{\bd}{\partial}
\newcommand{\cN}{\mathscr N}

\newcommand{\cE}{\mathcal E}

\newcommand{\BTop}{BTOP}
\newcommand{\BTopSpin}{BTOPSPIN}

\newcommand{\La}{\Lambda}
\newcommand{\Zpi}{\bZ[\pi]}

\DeclareMathOperator{\Hom}{Hom}
\DeclareMathOperator{\Ext}{Ext}
\DeclareMathOperator{\id}{id}

\newcommand{\BMb}{M(w_1, w_2)}
\newcommand{\BM}{E}

\calclayout

\begin{document}

\title{Quotients of $S^2\times{S^2}$}

\author{Ian Hambleton} 
\address{Department of Mathematics \& Statistics
 \newline\indent
McMaster University
 \newline\indent
Hamilton, ON  L8S 4K1, Canada}
\email{hambleton@mcmaster.ca}
\author{Jonathan A.~Hillman}
\address{
School of Mathematics and Statistics
\newline\indent University of Sydney
\newline\indent
Sydney,  NSW 2006, Australia }
\email{jonathan.hillman@sydney.edu.au}

\thanks{This research was partially supported by an NSERC Discovery Grant.}

\begin{abstract}
We consider closed topological 4-manifolds $M$ with universal cover ${S^2\times{S^2}}$ and Euler characteristic $\chi(M) = 1$. 
All such manifolds with $\pi=\pi_1(M)\cong{\cy 4}$ are homotopy equivalent.
In this case, we show that there are four homeomorphism types, and
propose a candidate for a smooth example which is not homeomorphic 
to the geometric quotient. If $\pi\cong \cy 2 \times \cy 2$, 
we show  that there are three homotopy types (and between 6 and 24 homeomorphism types).
\end{abstract}


\maketitle

\section{Introduction}

The goal of this paper is to characterize 4-manifolds with universal cover 
$S^2\times{S^2}$ up to homeomorphism in terms of standard invariants,
continuing the program of  \cite[Chapter 12]{Hillman:2002}.
Our  approach combines the analysis of Postnikov sections with recent results in surgery. The main new ingredient is the use of  bordism calculations to study the difference between homotopy self-equivalences and homeomorphisms of these $4$-manifolds.

A 4-manifold $M$ has universal covering space $\widetilde{M}\cong{S^2}\times{S^2}$ 
if and only if $\pi=\pi_1(M)$ is finite, $\chi(M)|\pi|=4$ and its Wu class $v_2(M)$
is in the image of $H^2(\pi;\mathbb{F}_2)$.
There are eight such manifolds which are geometric quotients,  in which $\pi$ acts through a subgroup of $\Isom(\mathbb{S}^2\times\mathbb{S}^2) = (O(3) \times O(3)) \rtimes \cy 2$  (see \cite[Chapter 12, \S2]{Hillman:2002}). 

\medskip
 Our classification results for the cases where $|\pi|=4$ are based on a detailed study of the intermediate coverings where $|\pi| \leq 2$ (see Sections \ref{sec:three}-\ref{sec:five}).

We first recall that  closed topological manifolds
with $\pi_1(M)=1$ or $\pi_1(M) = \cy 2$  have already been classified
(without assumption on the universal covering):
\begin{enumerate}
\item If $\pi=\cy n$, and $M$ is orientable, then $M$ is classified up to homeomorphism by
its intersection form on $H_2(M;\bZ)/Tors$, 
$w_2(M)$ and the Kirby-Siebenmann (KS) invariant
(see Freedman \cite{Freedman:1982} for $\pi =1$, 
and \cite[Theorem C]{Hambleton:1993a} for $\pi= \cy n$).
\item If $\pi=\cy 2$, and $M$ is non-orientable, 
then $M$ is classified up to homeomorphism by explicit invariants 
(see \cite[Theorem 2]{Hambleton:1994}), and a complete list of such manifolds
 is given in \cite[Theorem 3]{Hambleton:1994}.
\end{enumerate}

If we further impose the condition that   $\widetilde M = S^2 \times S^2$, 
 then it is convenient to separate the orientable and non-orientable cases.  
There are two orientable geometric $\cy 2$-quotients, 
namely the $2$-sphere bundles $S(\eta \oplus 2 \epsilon)$ and $S(3 \eta)$ 
over $RP^2$, where $\eta$ is the canonical line bundle over $RP^2$. 
The second manifold is non-spin and has a non-smoothable homotopy equivalent ``twin" 
$\ast M$ with KS $\neq 0$.

In the non-orientable case, there are two geometric  $\cy 2$-quotients: $S^2\times{RP^2}$ and $S^2\tilde\times{RP^2}=S(2\eta\oplus\epsilon)$, and 
one further smooth manifold $RP^4 \sharp_{S^1} RP^4$ obtained by removing a tubular neighbourhood of $RP^1 \subset RP^4$, 
and gluing two copies of the complement together along the boundary. Each of these has a homotopy equivalent twin $\ast M$ with KS $\neq 0$,
so there are six such non-orientable manifolds  
 (for these results see \cite[Chapter 12]{Hillman:2002} and  \cite{Teichner:1997}).

\begin{remark}
More generally, 
if  $\pi$ has order 
2 or 4 then $Wh(\pi)=0$ and the natural homomorphism 
from $L_4(1)$ to $L_4(\pi,-)$ is trivial (see Wall \cite[\S 3.4]{Wall:1976}).
Thus if $M$ is non-orientable we may surger the normal map $M\#E_8\to{M}\#S^4=M$ to obtain a  \emph{twin}:
that is a homotopy equivalent 4-manifold $\ast M$ 
with the opposite Kirby-Siebenmann invariant. Here $E_8$ denotes a closed, 
1-connected, topological $4$-manifold constructed by Freedman \cite{Freedman:1982}, whose  intersection form is definite of rank 8.
\end{remark}

%\smallskip
We now assume that $|\pi|=4$, which implies that 
$\chi(M)=1$  for  any quotient $M$ of $S^2 \times S^2$ by a free $\pi$-action. 
Any such $M$ must be non-orientable, 
since orientable closed 4-manifolds with finite fundamental group have 
Euler characteristic $\geq2$ (by Poincar\'e duality with $\mathbb{Q}$-coefficients).
If $\pi = \cy 4$, there is just one geometric quotient 
$\bM$ obtained from the free action generated by $(u,v) \mapsto (-v,u)$, for  $(u, v) \in S^2 \times S^2$. 

 \begin{thma} Let $N$ be a closed topological $4$-manifold with $\widetilde N =S^2 \times S^2$ and $\pi_1(N) = \cy 4$.
\begin{enumerate}
\item Each  $N$ is homotopy equivalent to the unique geometric quotient $\bM$. 
\item Every self homotopy equivalence of $\bM$ is homotopic to a self-homeomorphism.
\item There are four such manifolds up to homeomorphism, 
of which exactly two have non-trivial Kirby-Siebenmann invariant.
 \end{enumerate}
 \end{thma}
 
\begin{remark}
An analysis of one construction of the geometric example $\bM$ leads to
the construction of another smooth 4-manifold in this homotopy type,
which may not be homeomorphic to the geometric manifold (see Section \ref{sec:seven}).
\end{remark}

\begin{remark}
When $|\pi|=4$,  the mod 2 Hurewicz homomorphism
 $h\colon \pi_2(M) \to H_2(M;\cy 2)$
is trivial.
Hence pinch maps have trivial normal invariants, so do not provide 
``fake" self homotopy equivalences, meaning a self equivalence not homotopic to a homeomorphism  (see \cite[p.~420]{Cochran:1990}). 
We rule out other fake self equivalences for $\pi = \cy 4$  in Section \ref{sec:thirteen}. 
\end{remark}

In the remaining cases, where $\pi = \cy 2 \times \cy 2$, we classify the homotopy types of Poincar\'e $4$-complexes, and determine the homotopy types of closed manifolds.  
 We will use the notation $PD_4$-complex for a finite Poincar\'e duality complex of formal dimension 4 (see \cite[\S 1]{Wall:1967}).

\begin{thmb} 
There are two quadratic $2$-types of $PD_4$-complexes $X$ with 
$\chi(X)=1$ and $\pi_1(X)=  \cy 2 \times \cy 2$, and seven homotopy types in all. 
\begin{enumerate}
\item All such complexes have universal cover homotopy equivalent to ${S^2}\times{S^2}$.
\item The two quadratic $2$-types are represented by the total spaces of
the two $RP^2$-bundles over $RP^2$.
\item
A third homotopy type includes a smooth manifold $N$
with $RP^4\#_{S^1}RP^4$ as a double cover.
\item
The remaining homotopy types do not include closed manifolds.
\end{enumerate}
\end{thmb}

The  primary homotopy invariants of a finite $PD_4$-complex $X$ are its fundamental group 
$\pi:=\pi_1(X,x_0)$, and its second homotopy group $\pi_2(X)$ as a module over the integral group ring $\La:=\Zpi$.
The \emph{quadratic $2$-type} (introduced in \cite{Hambleton:1988}) is represented by the quadruple:
$$ \left [ \pi_1(X), \pi_2(X), k_X, s_X\right ]$$
where $s_X$ denotes
the equivariant intersection form $s_X\colon \pi_2(X) \times \pi_2(X) \to \La$,  and 
$$k_X \in H^3(\pi; \pi_2(X))$$
is the first $k$-invariant of  the algebraic $2$-type $[\pi_1(X), \pi_2(X), k_X]$
 as introduced by MacLane and Whitehead \cite{maclane-whitehead1}. 
 This data determines a space $P:=P_2(X)$, which 
 is a fibration over $K(\pi, 1)$,  classified by $k_X$,  with fibre $K(\pi_2(X), 2)$ and there is a $3$-connected reference map $\tilde c\colon X \to P$ lifting the classifying map $c\colon X \to K(\pi, 1)$ for the universal covering $\wX \to X$. Equivalently, $P_2(X)$ is the second stage of a Postnikov tower for $X$.

An \emph{isometry} of two such quadruples is an isomorphism on $\pi_1$, $\pi_2$ inducing an isometry of the equivariant intersection forms, and respecting the $k$-invariants.

\medskip
The first statement in Theorem B about the quadratic 2-types was proved in \cite[Chapter 12, \S6]{Hillman:2002}, but the homotopy classification is new. We use the invariants of \cite{Hambleton:1978} and \cite{Kim:1990} to determine which homotopy types contain closed manifolds. The homeomorphism classification appears difficult: 
all we can say at this stage is that in each case
the TOP structure set has 8 members, 
so that there are between 6 and 24 homeomorphism types of such manifolds,
of which half are not stably smoothable. 
To resolve this ambiguity, more information is needed about self homotopy equivalences. 

\medskip
Here is an outline of the paper.  After some preliminary material in Sections \ref{sec:two}-\ref{sec:twob}, 
we show that there are either two or 
four homeomorphism types with $\pi={\cy 4}$.  Part (i) of Theorem A is proved in Lemma \ref{lem:orderfour}.
We then  review the constructions of the non-orientable smoothable quotients of $S^2\times{S^2}$
with $\pi=\cy 2$ (see Sections \ref{sec:three}-\ref{sec:five}).

In Section \ref{sec:eight} we construct a new smooth 4-manifold $N$ in the  quadratic $2$-type
of the bundle space $RP^2\tilde\times{RP^2}$, 
but distinguished from it by its non-orientable double covers (see Definition \ref{def:N}).
In particular, N is not a geometric quotient.
In Sections \ref{sec:nine}-\ref{sec:ten} we show that there are no other homotopy types of 4-manifolds
with $\pi= \cy 2 \times \cy 2$ and $\chi=1$. This completes the proof of Theorem B.

In Section \ref{sec:thirteen} we complete the proof of  Theorem A via a stable homeomorphism classification result.
In  Section \ref{sec:seven} we construct a smooth manifold with $\pi = \cy 4$, which may not be diffeomorphic or even homeomorphic 
to the geometric quotient (see Definition \ref{def:elevenone}). 
The same strategy does not seem to provide a candidate for a smooth fake $RP^2\times{RP^2}$.

\tableofcontents
\begin{acknowledgement} The authors would like to thank Larry Taylor for useful conversations,  and the referee for many valuable suggestions and comments.
JAH would like to thank the Department of Mathematics 
and Statistics at McMaster University 
and the Fields Institute for their hospitality and support in September and October 2009, and in October 2017, respectively.
\end{acknowledgement}

\section{The structure set}\label{sec:two}
Classical surgery theory studies the \emph{structure set} $S_{TOP}(M)$, which consists of pairs $(N,f)$ of closed $4$-manifolds $N$ and a homotopy equivalence $f\colon N \to M$, modulo those homotopic to homeomorphisms. 
\emph{Here and throughout the paper we will always work with pointed spaces and base-point preserving maps}. 

 If $M$ is non-orientable, let $w\colon \pi_1(M) \to \cy 2$ denote the orientation character given by the first Stiefel-Whitney class $w_1(M) \in H^1(M; \cy 2)$. We fix a local coefficient system $\{\bZ^w\}$  induced by the classifying map  of the orientation double cover, and use it to define the homology of $M$ with ``twisted" coefficients. A choice of generator  $[M] \in H_4(M;\bZ^w) \cong \bZ$ gives a fundamental class for Poincar\'e duality (see Wall \cite[Chapter 1]{Wall:1967} and Taylor \cite[\S 5]{Taylor:2008}).

\medskip
The surgery exact sequence
$$ \dots \to L_5(\pi, w) \to S_{TOP}(M) \to [M, G/TOP] \to  L_4(\pi, w) $$
leads to a computation of $S_{TOP}(M)$ in favourable circumstances. The general theory due to Browder,  Kervaire, Milnor, Novikov, Sullivan and Wall for high-dimensional smooth or PL manifolds (see \cite{Wall:1999}) was extended to topological manifolds by Kirby and Siebenmann \cite{Kirby:2001}, 
and to $4$-manifolds with \emph{good} fundamental groups  by Freedman \cite{Freedman:1982}. In particular, surgery theory ``works" for topological $4$-manifolds with finite fundamental group.  We refer the reader to Kirby and Taylor \cite{Kirby:2001} for an overview of surgery theory in low dimensions.

In our situation, it is not difficult to compute the size of the structure set  $S_{TOP}(M)$. The remaining obstacle to obtaining a homeomorphism classification is to understand the action of homotopy self equivalences on the structure set. 

 Note that $G/TOP$ inherits an $H$-space structure as the degree zero space of  the connective $\bL_0$-theory spectrum, This $H$-space structure induces an  alternate abelian group structure on 
 $[X, G/TOP]$,  for any closed topological $4$-manifold,  distinct from the usual Whitney sum structure from bundle theory. 
 With this structure, Poincar\'e duality with $\bL_0$-theory coefficients  gives an isomorphism
 $$[X,G/TOP]= H^0(X; \bL_0) \cong H_4(X; \bL_0^w)$$
 of abelian groups.

Since $\pi_i(G/TOP)=0$ in all odd dimensions
and the first significant $k$-invariant of $G/TOP$ is 0, 
there is a 6-connected map $G/TOP\to{K(\bZ/2,1)}\times{K(\bZ,4)}$
(see \cite[\S 2]{Kirby:2001}).
 Hence,  in these low dimensions, 
$$[X,G/TOP]\cong{H^2(X;\bZ/2)}\oplus{H^4(X;\bZ)}.$$
It follows that this isomorphism is compatible with 
 $\bL_0$-theory Poincar\'e duality on $[X, G/TOP]$, and with ordinary Poincar\'e duality on the right-hand side,
 induced by cap product with
a (twisted) fundamental class
$$[X] \in H_4(X; \bZ^w) \cong \bZ$$
where $w=w_1(X)$.

We can now determine the size of $S_{TOP}(M)$ for manifolds with $\wM = S^2 \times S^2$ and fundamental groups of order four.

\begin{theorem}\label{thm:twoonet} Let $M$ be a closed topological $4$-manifold with $\pi_1(M) = \cy  4$ and $\chi(M)=1$.
The structure set $S_{TOP}(M)$ has four members, 
and there are either two or four homomeomorphism types 
of manifolds homotopy equivalent to $M$.
\end{theorem}

\begin{proof}
The normal invariant map in the surgery exact sequence
$$ S_{TOP}(M) \to [M,G/TOP] \cong  H^2(M;\cy 2) \oplus H^4(M;\bZ)$$ is a bijection, 
since the groups $L_5(\cy 4,-)$ and $L_4(\cy 4,-)$ are both zero (see Wall \cite[Theorem 3.4.5]{Wall:1976}). 
 Recall that $\chi(M) =1$ implies that $M$ is non-orientable, so the surgery obstruction groups denoted $L_*(\cy 4,-)$ appear with non-trivial orientation character.
The cohomology groups $H^2(M;\cy 2) = \cy 2$ and  $H^4(M;\bZ)= \cy 2$ were computed in \cite[Chapter 12, \S4]{Hillman:2002}.
Hence $|S_{TOP}(M)|=4$.
As observed in the Introduction, every such manifold $N$ has a fake twin $\ast N$.
\end{proof}

\begin{remark}
In particular,
if $h\colon M' \to M$ is a homotopy equivalence with  nontrivial normal invariant $\eta(h) \in H^2(M; \cy 2)$, then every closed 4-manifold 
with $\pi={\cy 4}$ and $\chi=1$ is homeomorphic to one of $M$, 
$M'$, $*M$ or $*M'$. The normal invariant of $M\, \sharp \, E_8 \to M$ is non-trivial in $H^4(M; \bZ) = \cy 2$. After surgery, this produces the twin manifold $\ast M$. 

Similarly, we have the manifold $\ast M'$ whose normal invariant is non-trivial in both summands of $[M,G/TOP]$,
and $KS(\ast M') = 0$ by the formula on \cite[p.~398]{Kirby:2001}. In contrast,  both $M'$ and $\ast M$
 have non-trivial Kirby-Siebenmann invariant. We do not know whether $\ast M'$ admits a smooth structure (see Section \ref{sec:seven} for a candidate).

In general, the normal invariant is an invariant of a map.
However, in this case we will complete the proof of Theorem A by showing that the homotopy type and the Kirby-Siebenmann invariant distinguish 
 homeomorphism types completely (see Section \ref{sec:thirteen}).
 \end{remark}
 The cases where $\pi_1(M) = \cy 2 \times \cy 2$ are similar.
 
 \begin{theorem}\label{thm:twotwot} Let $M$ be a closed topological $4$-manifold with
  $\pi_1(M) =\cy 2 \times \cy 2$ and $\chi(M)=1$.
The structure set $S_{TOP}(M)$ has eight members, consisting of up to 
 four distinct twin pairs of homomeomorphism types $(N, \ast N)$
of manifolds homotopy equivalent to $M$.
\end{theorem}
\begin{proof}
 We have the cohomology groups $H^2(M;\cy 2) = (\cy 2)^3$ and $H^4(M;\bZ) =\cy 2$. 
Moreover, (i) the map $S_{TOP}(M) \to [M, G/TOP]$ is injective, since $L_5((\bZ/2)^2,-)=0$, and (ii)
the surgery obstruction map
from $[M,G/TOP]$ to $L_4((\bZ/2)^2,-)=\cy 2$ is onto.  
  Hence  $S_{TOP}(M)$ has 8 elements $(N,f)$ for each homotopy type of such manifolds $M$, with domains consisting of 4 twinned pairs $(N, \ast N)$  
(see \cite[Chapter 12, \S7]{Hillman:2002}).
\end{proof}

\begin{remark}
Half of the elements of   $S_{TOP}(M)$ have domains with nontrivial Kirby-Siebenmann invariant, 
and so the image of $\Homeo (M)$ in the group of (free homotopy classes of)
self homotopy equivalences of $M$ has index at most 4.
However, whether every self homotopy equivalence of $M$ 
is homotopic to a homeomorphism remains  open.
To make further progress we need explicit representatives for the self homotopy equivalences.
\end{remark}

\section{Homotopy type invariants for finite $PD_4$-complexes}\label{sec:twob}
Let $B = P_2(X)$ denote the   Postnikov $2$-section of a finite Poincar\'e $4$-complex $X$  with orientation character $w\colon \pi_1(X) \to \cy 2$. A $B$-polarized $PD_4$-complex consists of a pair
$(X, f)$, where $f\colon X \to B$ is a $3$-equivalence. 
Two such pairs $(X, f)$ and $(Y,g)$ are equivalent if there exists a homotopy equivalence 
$h \colon X \to Y$ 
such that $f \simeq g\circ h$. Following \cite[\S 1]{Hambleton:1988}, we let  $S_4^{PD}(B, w)$ denote the set
of homotopy types of  $B$-polarized $PD_4$-complexes.

 For $PD_4$-complexes with finite fundamental group,
the set   $S_4^{PD}(B,w)$ is determined by the quadratic 2-type and 
a secondary invariant depending on $\pi_2(X)$ as a $\pi_1(X)$-module.  
Let $S_4^{PD}(B,w, \lambda)$ denote the subset of $S_4^{PD}(B, w)$ of $B$-polarized Poincar\'e complexes $(X,f)$, such that  $\lambda\colon \pi_2(B) \times \pi_2(B) \to \La$ is a hermitian form which is mapped to the intersection form $s_X$ via $f^*$ 
 (see \cite[p.~357]{Kasprowski:2021b}).
Note that if $\pi_2(B)\not=0$ then $w$ is determined by $\lambda$
 \cite [Chapter 1.4]{Teichner:1992}. 
 The elements of  $S_4^{PD}(B,w, \lambda)$  are called $PD_4$-polarizations 
of the quadratic $2$-type.

 In the rest of the paper, we will always assume that a $PD_4$-complex $X$ has one top cell (see Wall \cite[Corollary 2.3.1]{Wall:1967}). In the following statement, $\Gamma_W(\pi_2(B))$ denotes Whitehead's quadratic functor. An action of  the torsion subgroup of  $\bZ^w\otimes_{\Zpi}\Gamma_W(\pi_2(B))$ on an element $(X,f)\in S_4^{PD}(B, w)$ is defined by writing $X = K \cup_g D^4$ with $K$ a $3$-complex, and re-attaching the top cell by a suitable element $\alpha \in \pi_3(K)$ (see 
 \cite[\S 1]{Hambleton:1988} or  \cite[p.~364]{Kasprowski:2021b} for the details of this construction).

\begin{theorem}\label{thm:quad2}
Each homotopy type within the quadratic $2$-type of a $PD_4$-complex $X$ 
with $\pi$ finite may be obtained by varying 
the attaching map of the top cell to the $3$-skeleton $X^{(3)}$. 
The torsion subgroup of $\bZ^w\otimes_{\Zpi }\Gamma_W(\pi_2(B))$ 
acts  freely and transitively on the set of $PD_4$-polarizations of the quadratic $2$-type.
\end{theorem}
\begin{proof} This result is due to Hambleton and Kreck \cite[Theorem 1.1]{Hambleton:1988}, 
Teichner \cite[Chap.~2]{Teichner:1992}, and Kasprowski and Teichner \cite[Theorem 1.5]{Kasprowski:2021b}. 
\end{proof}

\begin{remark} 
In particular, the cardinality of this torsion subgroup is an upper bound for the number 
of homotopy types within the quadratic 2-type.
The homotopy types correspond bijectively to the orbits of the group  
$ \Aut(B, w,\lambda)$ of homotopy classes of self-homotopy equivalences of $B$ 
which preserve the orientation character $w$ and the hermitian form $\lambda$
on $S_4^{PD}(B,w, \lambda)$.

The quadratic 2-types of interest to us are those of the non-orientable quotients 
of $S^2\times{S^2}$. For the $\cy 2$-quotients, the number of distinct homotopy types is
determined by explicit constructions based on Theorem 3.1 (see Proposition \ref{prop:s2xrp2}),  and distinguished by explicit invariants in Theorem \ref{thm:bundlespaces}.

 However, for the quotients with $\pi=\cy 2 \times \cy 2$,  we need the additional information  provided by the action of $ \Aut(B, w,\lambda)$ 
to completely analyse the number of distinct homotopy types
(see Propositions \ref{prop:rp2xrp2} and \ref{prop:rp2txrp2}).

We note  that the results from \cite[Chapter 12]{Hillman:2002} 
which we cite are formulated there in terms of closed 4-manifolds,
but apply equally well to $PD_4$-complexes.
\end{remark}

\medskip
We  first specialize to the cases where $\pi_1(X) = \cy 4$. 

\begin{lemma}\label{lem:orderfour} 
Every $PD_4$-complex $X$ with $\pi_1(X)={\cy 4}$ and $\chi(X)=1$ 
is homotopy equivalent to the geometric quotient $\bM$. Moreover, the image $c_*[X] \in H_4(\pi; \bZ^w)$ of its fundamental class is non-zero.
\end{lemma}
\begin{proof} 
The universal cover $\widetilde{X}$ is homotopy equivalent to $S^2\times{S^2}$
\cite[Lemma 12.3]{Hillman:2002}.
It is shown in \cite[Chapter 12, \S6]{Hillman:2002}) that 
the quadratic $2$-type is uniquely determined by  the assumptions on $X$,  and  the torsion subgroup of $\bZ^w\otimes_{\Zpi }\Gamma_W(\pi_2(X))$ is zero. 
For the last statement, note that the group $\pi=\cy 4$ acts on $\Pi:= \pi_2(X)=\bZ^2$ via
$\left(\begin{smallmatrix}
\hphantom{-}0&1\\-1&0\end{smallmatrix}\right)$,
and so $\Pi\cong\Lambda/(t^2+1)=\bZ[i]$. The general result of \cite[Theorem 1.10]{Kasprowski:2022} is a stable exact sequence
$$ \cE: 0 \to H_2(K;\La^w)  \to \pi_2(X) \oplus \La^r \to H^2(K; \La^w) \to 0$$
where $K$ is any finite $2$-complex with $\pi_1(K) = \pi$, and the extension class 
$$[\cE] \in \Ext^1_{\La}(H^2(K;\La), H_2(K;\La))$$
can be naturally identified with 
the image $c_*[X] \in H_4(\pi;\bZ^w)$  of the fundamental class of $X$. Since $H_1(\pi; \bZ[i]) = 0$ and $H_1(\pi; H_2(K; \La^w)) \cong  H_4(\pi; \bZ^w) = \cy 2$, the extension $\cE$ must be non-split. 
\end{proof}

\medskip
Finally, we recall two additional invariants which can be used to show that not every finite $PD_4$-complex  is homotopy equivalent to a closed manifold. 
\begin{example}\label{ex:threeone}
Kim, Kojima and Raymond  \cite{Kim:1992} defined a $\cy 4$-valued quadratic function 
$q_{KKR}(M)$ on $\pi_2(M)\otimes{\cy 2}$, 
for $M$ a closed non-orientable 4-manifold, by the rule
\[
q_{KKR}(M)(x)=e(\nu(S_x))+2|\Self(S_x)|,
\]
where $S_x\colon S^2\to{M}$ is a self-transverse immersion representing $x$,
$e(\nu(S_x))$ is the Euler number of the normal bundle
and $\Self(S_x)$ is the set of double points of the image of $S_x$.
This is a  \emph{quadratic  enhancement} of the mod 2 equivariant intersection pairing
on $\widetilde{M}$, 
and is a homotopy invariant for $M$.
\end{example}
The second invariant is an obstruction to the reducibility of the Spivak normal fibre space to a vector bundle.

\begin{example}\label{ex:threetwo}
Let $X' \to X$ be a double cover of finite $PD_{2n}$-complexes, classified by a map $f\colon X \to RP^{k+1}$, for some $k >>n$. Following Hambleton and Milgram \cite{Hambleton:1978}, we say that the double covering is \emph{Poincar\'e splittable} if the homotopy class of the map $f$ contains a representative which is
Poincar\'e transverse to $RP^k \subset RP^{k+1}$. This always holds if $X' \to X$ is a double cover of closed manifolds, or more generally if the Spivak normal fibre space has a vector bundle reduction.
There is a quadratic map
$$q\colon H^n (X'; \cy 2) \to \cy 2$$
refining the non-singular bilinear form
$$\ell(a,b) = \la a\cup T^*b, [X']\ra,$$
where $a, b \in H^n(X';\cy 2)$ and $T \colon X' \to X'$ is the free involution induced by the double cover. 
Let $A(X,f) \in \cy 2$ denote the Arf invariant of 
this quadratic form. Then $A(X,f)$ defines a homomorphism $\cN_{2n}(RP^{\infty}) \to \cy 2$, 
which vanishes for double covers of manifolds (see \cite[Proposition 2.1]{Hambleton:1978}). 
If $X$ is orientable, then $A(X,f) = 0$ for any double cover (see \cite{Hambleton:2019}),
 but there exist non-orientable double covers in each even dimension $\geq 4$ for which $A(X,f) \neq  0$ (see \cite[Theorem 3.1]{Hambleton:1978}).
\end{example}

\section{Non-orientable quotients of $S^2\times{S^2}$ with $\pi=\cy 2$}\label{sec:three}

We introduce some notation for later use. Let $A$ be the antipodal involution of $S^2$, and let  $\eta\colon S^3\to{S^2}$ denote the Hopf fibration.
Let $\bar\eta\colon S^3\to{RP^2}$ be the composite of $\eta$ with the projection $S^2\to{RP^2}=S^2/\{x\sim{A(x)}\}$.
In this section we describe the homotopy types of non-orientable quotients of $S^2\times{S^2}$ by a free involution.

\begin{proposition} 
\label{prop:s2xrp2}
Let $X$ be a finite non-orientable $PD_4$-complex with $\pi_1(X) = \cy 2$. 
If $\wX \simeq S^2 \times S^2$, then
\begin{enumerate}
\item $X$ has the quadratic $2$-type of $S^2\times{RP^2}$.
\item There are four distinct homotopy types of $PD_4$-complexes in this quadratic 
$2$-type.
\item Exactly three of these homotopy types are represented by closed manifolds.
\end{enumerate}
The manifolds in this quadratic $2$-type are $S^2\times{RP^2}$, $S^2\tilde\times{RP}^2$ 
and  $RP^4\#_{S^1}RP^4$.
\end{proposition}

\begin{proof}
Since $\chi(\widetilde{X})=4$ and $\pi_1(X)=\cy 2$, we have $\chi(X)=2$.
There are two quadratic 2-types of non-orientable $PD_4$-complexes $X$ with $\pi=\cy 2$ and $\chi(X)=2$.
Moreover, all such quotients of $S^2\times{S^2}$ have the quadratic 2-type of $S^2\times{RP^2}$ (see  \cite[Chapter 12, \S6]{Hillman:2002}). 
We now apply Theorem \ref{thm:quad2} to analyse the homotopy types.

Let $K=\overline{S^2\times{RP^2}\setminus{D^4}}$ be the 3-skeleton of $S^2\times{RP^2}$,
let $I_1, I_2\colon S^2\to\widetilde{K}=\overline{S^2\times{S^2}\setminus2D^4}$
be the inclusions of the factors, 
and let $[J]$ be the homotopy class of a fixed lift $\widetilde{J}\colon S^3\to\widetilde{K}$
of the natural inclusion  $J\colon S^3=\partial{D}^4\to{K}$.

Since $\pi_2(S^2\times{RP^2})=\bZ^2$ is generated by 
$I_1$ and $I_2$,
the group $\Gamma_W(\pi_2(S^2\times{RP^2}))$ has basis 
$[I_1,I_2]$,  $\eta_1=I_1\circ \eta$, and $\eta_2 =I_2\circ \bar\eta$.
Since the nontrivial element of $\pi$ fixes $I_1$ and changes the sign of $I_2$,
it fixes $\eta_1$ and $\eta_2$ and changes the sign of $[I_1,I_2]$.
Hence $\Gamma_W(\pi_2(S^2\times{RP^2}))\cong\bZ^w\oplus\bZ^2$,
and so 
$\bZ^w\otimes_{\Zpi}\Gamma_W(\pi_2(S^2\times{RP^2}))
\cong\bZ\oplus(\cy 2)^2$.
In particular,
the torsion subgroup of $\bZ^w\otimes_{\Zpi}\Gamma_W(\pi_2(S^2\times{RP^2}))$ 
is isomorphic to $(\cy 2)^2$, and is generated by the images of $\eta_i$, 
for $i=1,2$.

Thus there are at most four homotopy types,  
represented by the $PD_4$-complexes $W_\alpha=K\cup_{[J]+\alpha}e^4$ 
corresponding to $\alpha=0$, $\eta_1$,  $\eta_2$ and $\eta_1+\eta_2$. 
Clearly $W_0=K\cup_{[J]}D^4\simeq{S^2\times{RP^2}}$.
According to \cite[p.~80]{Kim:1992}, $W_{\eta_1}\simeq{S^2\tilde\times{RP}^2}$
and $W_{\eta_1+\eta_2}\simeq{RP^4\#_{S^1}RP^4}$.
We shall describe these manifolds explicitly in the next section,
and show that they have distinct homotopy types in  Theorem \ref{thm:bundlespaces}.
In \cite{Hambleton:1978} it is shown that the $PD_4$-complex $P_{HM}=W_{\eta_2}$ 
is not homotopy equivalent to a closed 4-manifold 
(note that \cite{Hambleton:1978} writes the factors in the opposite order).
Thus these four homotopy types are distinct  and part (iii) follows.
\end{proof}

\begin{remark}
The only other quadratic 2-type with $\pi=\cy 2$, $w_1\not=1$ and $\chi=2$ 
is that of $RP^4\#CP^2$ (the nontrivial $RP^2$-bundle over $S^2$), which contains two homotopy types.
One of these is not homotopy equivalent to a closed 4-manifold, by \cite[\S3.3.1]{Teichner:1992}.
These $PD_4$-complexes   have universal cover $\wX\simeq{S^2\tilde\times{S^2}}$,
and do not cover $PD_4$-complexes with $\chi=1$
(see \cite[ Lemma 12.3]{Hillman:2002}).
\end{remark}

\section{Explicit constructions for $S^2\tilde\times{RP^2}$ and $RP^4\#_{S^1}RP^4$}\label{sec:four}
The goal of this section is to express these two smooth model manifolds in terms of explicit building blocks. 
The ``coordinate" formulas
will be used in later sections to compute homotopy type invariants, and to construct smooth model manifolds with $\chi(M) =1$.

Let $E$ be a regular neighbourhood of $RP^2=\{[x{ : }y{ : }z{ : }0{ : }0]\mid{x^2+y^2+z^2=1}\}$ in $RP^4$, and note the following properties:
\begin{enumerate}
\addtolength\itemsep{2pt}
\item $\nu=\overline{RP^4\setminus{E}}$ is a regular neighbourhood of $RP^1=\{[0{ : }0{ : }0{ : }u{ : }v]\mid{u^2+v^2=1}\}$;
\item $\partial{E}=\partial\nu$ is both the total space of 
a non-trivial $S^1$-bundle over $RP^2$
and the mapping torus ${S^2}\tilde\times{S^1}=S^2\times[0,1]/(s,0)\sim(A(s),1)$;
\item In particular, $\pi_1(\partial{E})\cong\bZ$,
and so $E$ is not the product $RP^2\times{D^2}$;
\item On passing to the universal cover we see that
$S^4=\widetilde{E}\cup\widetilde\nu$;
\item We may assume that $\widetilde{E}=\{(x,y,z,u,v)\in{S^4}\mid{u^2+v^2\leq\frac14}\}. $
\end{enumerate}

\medskip
Now let $h\colon \widetilde{E}\to{S^2\times{D^2}}$ be the homeomorphism given by 
$h(\tilde{e})=(x/r,y/r,z/r,2u,2v)$, where $r=\sqrt{x^2+y^2+z^2}$, 
for all $\tilde{e}=(x,y,z,u,v)\in\widetilde {E}$. 
It follows that we may write $E=S^2\times{D^2}/(s,d)\sim(A(s),-d)$,
and the projection $p\colon E\to{RP^2}$ is then given by $p([s,d])=[s]\in{RP^2}$.
The space $E$ is also an orbifold bundle with general fibre $S^2$ over the marked disc $D(2)$,
via the projection $p'([s,d])=d^2$.
Here we view $D^2$ as the unit disc in the complex plane.

We shall view $S^2$ henceforth as the purely imaginary quaternions 
of length 1.
The antipodal map $A$ is multiplication by $-1$, 
while conjugation by $\mathbf{k}$ induces rotation $R_\pi$ through
a half-turn about the $\mathbf{k}$-axis.
The sphere is the union of two hemispheres $S^2=D_-\cup{D_+}$ with
boundary $S^1=D_-\cap{D_+}$ in the $(\mathbf{i},\mathbf{j})$-plane.

The orthogonal projection $\lambda$ of the purely imaginary quaternions 
onto the $(\mathbf{i},\mathbf{j})$-plane restricts to homeomorphisms 
from each of $D_-$ and $D_+$ onto the unit disc in this plane,
and $\lambda((R_\pi(s))=A(\lambda(s))=-\lambda(s)$, for all $s\in{S^2}$.

\begin{definition}[Construction of $S^2\tilde\times{RP^2}$]\label{def:fiveone}
Doubling $E$ along its boundary gives the total space of an $S^2$-bundle over $RP^2$.
This space $DE$ is non-orientable and $v_2(DE)\not=0$,
since the core $RP^2$ in $E$ has self-intersection 1 (mod 2).
Thus $DE$ is the nontrivial, non-orientable $S^2$-bundle space 
\[
S^2\tilde\times{RP^2}=S^2\times{S^2}/(s,t)\sim(A(s),R_\pi{t}).
\]
Composition of the double covering of $RP^2$ with 
the projection of $S^2\times{S^2}$ onto its first factor induces 
the $S^2$-bundle projection $DE\to{RP^2}$. 
\end{definition}
The space $DE$ is also the total space of an orbifold bundle 
with general fibre $S^2$ over the orbifold $S(2,2)$ (the double of $D(2)$).

We may construct a different 4-manifold by identifying two copies of $E$ 
via a diffeomorphism of their boundaries which does not extend across $E$.
The action of conjugation by $e^{\pi{\mathbf{i}}t}$ on $S^2$ 
inside the unit quaternions is rotation through $2\pi{t}$ radians 
about the $\mathbf{i}$-axis.
\begin{definition}[Construction of $RP^4\#_{S^1}RP^4$] \label{def:twist}
Let $E_1$ and $E_2$ be two copies of $E$, 
and let $\xi\colon \partial{E_1}\to\partial{E_2}$ be the map given by
\[
\xi([s,y^2\mathbf{i}]_1)=[y\mathbf{i}s(y\mathbf{i})^{-1},y^2\mathbf{i}]_2,
~\forall~s\in{S^2},~\forall~{y=e^{\pi{\mathbf{k}}t}}, ~0\leq{t}\leq1.
\]
We define $RP^4\#_{S^1}RP^4=E_1\cup_\xi{E_2}$ (see \cite[p.~651]{Hambleton:1994} for another description).
\end{definition}

%\
Note that $e^{\pi{\mathbf k}t}$ is a square root for $e^{2\pi{\mathbf{k}}t}$. 
This ``twist map" $\xi$ does not extend to a homeomorphism 
from $E_1$ to $E_2$ (see \cite[Corollary 2.2]{Kim:1990}).

\begin{remark}
The complication in the formula for $\xi$ in Definition \ref{def:twist} flows from the fact 
that this copy of $S^1$ is not closed under quaternionic multiplication,
whereas its translate $S^1\bf{i}$ is the unit circle in
$\mathbb{R}\oplus\mathbb{R}\mathbf{k}\cong\mathbb{C}$. 
\end{remark}

\begin{remark}\label{rem:fivesix}
The manifold $RP^4\#_{S^1}RP^4$ is the total space of an orbifold bundle with regular fibre $S^2$
over $S(2,2)$.
The exceptional fibres are the cores $RP^2$ of the copies of $E$,
and each has self-intersection 1.
Hence $v_2(RP^4\#_{S^1}RP^4)\not=0$.
We shall show in the next section that $RP^4\#_{S^1}RP^4$ is \emph{not} homotopy equivalent 
to a bundle space \cite{Kim:1992}, and hence it is not geometric.
\end{remark}

\medskip
We conclude this section with an explicit identification of $\wX \simeq S^2 \times S^2$ for the model manifold $X= RP^4\#_{S^1}RP^4$.

\medskip
The universal cover of $RP^4\#_{S^1}RP^4$ is the union 
$\widetilde{E}_1\cup_{\tilde\xi}\widetilde{E}_2$,
where $\tilde\xi$ is the lift of $\xi$ given by $\tilde\xi((s,x)_1)=(xsx^{-1},x)_2$, 
for all $(s,x)\in{S^2}\times{S^1}=\partial\widetilde{E}_1$.
Let $\mu_t(x)=\cos(\frac\pi2t)\mathbf{1}+\sin(\frac\pi2t)x$,
for $x\in{S^1}$ and $0\leq{t}\leq1$.
Then $\mu_0(x)=\mathbf{1}$ and $\mu_1(x)=x$, for all $x\in{S^1}$,
and 
\[
\tilde\xi_t((s,x)_1)=(\mu_t(x)s\mu_t(x)^{-1},x)_2
\] 
defines an isotopy from the identity to $\widetilde\xi$.
Hence $\widetilde{E}_1\cup_{\tilde\xi}\widetilde{E}_2\cong{S^2}\times{S^2}$.

\medskip
We may make this explicit as follows.
Let $P(r,x)=\sin(\frac\pi2r)x+\cos(\frac\pi2r){\mathbf{k}}$,
for $0\leq{r}\leq1$ and $x\in{S^1}=D_-\cap{D_+}$.
Then $P(0,x)=\mathbf{k}$ and $P(1,x)=x$, for all $x\in{S^1}$.
Let $V\colon D_+\to{S^3}$ be the function defined by 
$V(d)=P(r,x)$ if $\lambda(d)=rx$, with $0\leq{r}\leq1$ and $x\in{S^1}$.
Then the function
$H\colon S^2\times{S^2}\to\widetilde{E}_1\cup\widetilde{E}_2$,
defined by
\eqncount
\begin{equation}\label{eq:H1}
H(s,d)=(s,d)_1\in\widetilde{E}_1,
~\forall~(s,d)\in{S^2}\times{D_-} 
\end{equation}
and
\eqncount
\begin{equation}\label{eq:H2}
H(s,d)=(V(d)sV(d)^{-1},d)_2\in\widetilde{E}_2,
~\forall~(s,d)\in{S^2}\times{D_+},
\end{equation}
is a homeomorphism.
Hence $RP^4\#_{S^1}RP^4\cong{S^2\times{S^2}}/\langle\psi\rangle$,
where $\psi$ is the free involution given by
\[
\psi(s,d)=(A(s),R_\pi(d)),~\forall~(s,d)\in{S^2}\times{D_-},
\]
and
\[
\psi(s,d)=(V(R_\pi(d))^{-1}V(d)A(s)V(d)^{-1}V(R_\pi(d)),R_\pi(d)),
~\forall~(s,d)\in{S^2}\times{D_+}.
\] 
It is clear from the formula that $\psi$ is an involution, since $R_\pi^2=Ix$.

If we set $x=\cos(2\pi{t})\mathbf{i}+\sin(2\pi{t})\mathbf{j}$ for some 
$0\leq{t}\leq1$ then we may write the factor $V(R_\pi(d))^{-1}V(d)$ 
more explicitly as
\[
V(R_\pi(d))^{-1}V(d)=\cos(\pi{r})\mathbf{1}+
\sin(\pi{r})\sin(2\pi{t})\mathbf{i}-\sin(\pi{r})\cos(2\pi{t})\mathbf{j}.
\]
Thus $V(R_\pi(d))^{-1}V(d)=\mathbf{1}$ when $r=0$ and $V(R_\pi(d))^{-1}V(d)=-\mathbf{1}$ when $r=1$.
(This expression was found by solving the linear system
\[
V(R_\pi(d))(u\mathbf{i}+v\mathbf{j}+w\mathbf{k}+z\mathbf{1})=V(d),
\]
for the unknowns $u,v,w,z\in\mathbb{R}$.)

\section{Distinguishing the homotopy types}\label{sec:five}

We shall follow \cite{Kim:1992} in using the mod 2  intersection pairing 
(in the guise of $v_2$) and the invariant $q_{KKR}$
to show that $RP^4\#_{S^1}RP^4$ 
is not homotopy equivalent to either of the $S^2$-bundle spaces.
As our construction of $RP^4\#_{S^1}RP^4$ differs slightly from that of \cite{Kim:1992},
we shall give details of the geometric computation of $q_{KKR}$ for these manifolds.

\begin{theorem} \label{thm:bundlespaces}
The model manifolds $S^2\times{RP^2}$, 
$S^2\tilde\times{RP}^2$ or $RP^4\#_{S^1}RP^4$ represent distinct homotopy types, distinguished by
the Wu class $v_2$ and the invariant $q_{KKR}$.
\end{theorem}

\begin{proof}
Let $M=S^2\times{RP^2}$, 
$S^2\tilde\times{RP}^2$ or $RP^4\#_{S^1}RP^4$,
and let $x, y\in\pi_2(M)$ be the classes corresponding 
to the first and second factors of $S^2\times{S^2}$.
Then $x+y$ corresponds to the diagonal.
In each case $x$ is represented by the (general) fibres 
of the (orbifold) bundle projections
to $RP^2$, $S(2,2)$ and $S(2,2)$, respectively,
which are embedded with trivial normal bundle,
and so $q_{KKR}(M)(x)=0$,
while the normal Euler number of the diagonal is $\pm2$.

Let $f\colon S^2\to{S^2}$ be the map given by $f(x,y,z)=(x,y,|z|)$ for all 
$(x,y,z)\in{S^2}$,
and let $g\colon S^2\to{RP^2}$ be the 2-fold cover.
The 2-sphere $\{(f(s),s)|s\in{S^2}\}\subseteq S^2\times{S^2}$ represents $y$, 
and has trivial normal bundle, since $f$ is null homotopic.
Its image in $S^2\times{RP^2}$ has a single double point, and so $q_{KKR}(S^2\times{RP^2})(y)\equiv 2 \pmod 4$.
The graph $\Gamma_g\subset{S^2}\times{RP^2}$ is an embedded 2-sphere
which lifts to the diagonal embedding in $S^2\times{S^2}$.
Since there are no self intersections,  $q_{KKR}(S^2\times{RP^2})(x+y)\equiv 2 \pmod 4$ also.
Hence $q_{KKR}(S^2\times{RP^2})$ is nontrivial for $S^2\times{RP^2}$.

In $S^2\tilde\times{RP}^2$ the fibre of the bundle projection to $RP^2$ represents $y$.
Hence $$q_{KKR}(S^2\tilde\times{RP}^2)(x)=q_{KKR}(S^2\tilde\times{RP}^2)(y)=0.$$
The image of the diagonal has a circle of self-intersections.
However $id_{S^2}$ is isotopic to a self-homeomorphism of $S^2$ 
which is the identity on one hemisphere and moves the equator off itself in the other hemisphere.
Hence the diagonal embedding  is isotopic to an embedding whose image has just one self-intersection.
Hence $q_{KKR}(S^2\tilde\times{RP}^2)(x+y)=0$ also,
and so $q_{KKR}(S^2\tilde\times{RP}^2)$ is identically 0 for $S^2\tilde\times{RP}^2$.

In $RP^4\#_{S^1}RP^4$ the class $y$ is represented by the image of
$\{\mathbf{j}\}\times{S^2}$.
Double points in the image correspond to pairs $\{s,s'\}\subset{S^2}$ such that
 $\psi(\mathbf{j},s)=(\mathbf{j},s')$.
 If $\{s,s'\}$ is such a pair then $s,s'\in{D_+}$, $s'=R_\pi(s)$ and 
 \[
 \mathbf{j}V(R_\pi(s))^{-1}V(s)=-V(R_\pi(s))^{-1}V(s)\mathbf{j}.
 \]
On using the explicit formula for $V(R_\pi(d))^{-1}V(d)$ given at the end of Section \ref{sec:four},
we see that we must have $\cos(\pi{r})=0$ and $\cos(2\pi{t})=0$.
Thus there are just two possibilities for $s$, differing by the rotation $R_\pi$.
We may check that the double point is transverse.
Hence $|\Self(S_y)|=1$.
Since $\{\mathbf{j}\}\times{S^2}$ has trivial normal bundle in $S^2\times{S^2}$,
$q_{KKR}(RP^4\#_{S^1}RP^4)(y)\equiv 2 \pmod 4$, and so $RP^4\#_{S^1}RP^4$ is not homotopy equivalent to $S^2\tilde\times{RP}^2$.
It is not homotopy equivalent to $S^2\times{RP^2}$ either, since $v_2(RP^4\#_{S^1}RP^4)\not=0$.
Thus these three manifolds may be distinguished by the invariants $v_2$ and $q_{KKR}$.
\end{proof}

We shall use the following simple observation in several places below.

\begin{lemma}
\label{lem:skeleton}
Let $X=K\cup{D^4}$ be a $PD_4$-complex with $3$-skeleton $K$ and one top cell. 
Then  $H^*(X;\mathbb{F}_2)\to{H^*(K;\mathbb{F}_2)}$ is a ring homomorphism which is an isomorphism in degrees $\leq3$.
\end{lemma}

\begin{proof}
This follows immediately from the fact that $H^k(X,K;\mathbb{F}_2)=0$ for $k\leq3$.
\end{proof}

For completeness, we show that $v_2(P_{HM})=0$.
Let $K=\overline{S^2\times{RP^2}\setminus{D^4}}$ be the 3-skeleton of $S^2\times{RP^2}$.
Since the homomorphisms $H^*(W_\alpha;\mathbb{F}_2)\to{H^*(K;\mathbb{F}_2)}$ 
are isomorphisms in degrees $\leq3$, $H^1(W_\alpha;\mathbb{F}_2)=\langle{x}\rangle$, 
where $x^3=0$ in all cases. 
Let $p\colon K\to{S^2}$ denote the restriction of the projection map to the first factor of $S^2\times RP^2$.
The group $H^2(K;\mathbb{F}_2)$ is generated by $x^2$
and the class $u$ pulled back by  $p\colon K\to{S^2}$. 
Since $p\circ \eta_2$ is a constant map, it follows that
$p\circ (J+\eta_2)=p\circ J$, which extends across $D^4$. Therefore
the map $p$ extends to a map from $P_{HM}$ to $S^2$, 
and so  $u^2=0$ in $H^4(P_{HM};\mathbb{F}_2)$. 
Also since $x^4=0$,  it follows that $v_2(P_{HM})=0$. 
On the other hand, this projection does not extend in this way when $\alpha=\eta_1$
or $\eta_1+\eta_2$, and in these cases $v_2\not=0$, as we have seen.

\section{$PD_4$-complexes with $\pi=(\cy 2)^2$ and $\chi=1$}\label{sec:eight}

We now consider the cases where $\pi = \cy 2\times \cy 2$. 
As mentioned in the Introduction, there are two geometric quotients, namely $RP^2\times{RP^2}$ 
and the non-trivial bundle $RP^2\tilde\times{RP^2}$. 
In this section, we will construct a third smooth manifold $N$ with universal cover 
$S^2\times{S^2}$  and fundamental group $\pi$, which is \emph{not} a geometric quotient.

Recall from Definition \ref{def:twist} that 
 the manifold 
$RP^4\#_{S^1}RP^4=E_1\cup_\xi{E_2}$ was expressed in terms of the glueing map
$\xi\colon \partial E_1 \to \partial E_2$.
We can define a smooth free involution $\theta$ 
of $\bd E_i = S^2\tilde\times{S^1}$, with quotient ${RP^2\times{S^1}}$, by the  map $\theta([s,x])=[-s,x]$.
Note that the maps $\theta$  and $\xi$ commute.
\begin{definition}\label{def:N}
Let $N$ denote the quotient space of $RP^4\#_{S^1}RP^4$
by the smooth free involution $F$ given by the formula
$F([s,d]_i)=[-s,d]_{3-i}$ for all
$[s,d]_i\in{E_i}$ and $i=1,2$.
\end{definition}
By construction, the  manifold $N$ has  $\pi_1(N) = \cy 2 \times \cy 2$ and $\chi(N) =1$. We summarize some of its properties. 

\begin{proposition}\label{prop:sevenN} The smooth closed $4$-manifold $N$ has the following properties:
\begin{enumerate}
\item  $N$ is the quotient of $RP^4\#_{S^1}RP^4$ by a smooth free  involution;
\item the universal covering $\wN = S^2 \times S^2$;
\item $N$ is in the quadratic $2$-type of $RP^2\tilde\times{RP^2}$; 
\item $N$ is not  a geometric quotient.
\end{enumerate}
\end{proposition}

\begin{proof} Part (i) is  immediate from Definition \ref{def:N},  and part (ii) follows from the construction of $RP^4\#_{S^1}RP^4$  given in Definition \ref{def:twist}.

Part (iv) follows from Theorem \ref{thm:bundlespaces}: the manifold $N$ is not homotopy equivalent to a geometric quotient (i,e,~a bundle space over $RP^2$), since
it is covered by $RP^4\sharp_{S^1}RP^4$,
which is not homotopy equivalent to a bundle space. 

In order to prove part (iii), we first collect some information about the quadratic $2$-types in this setting.
In \cite[Chapter 12, \S5]{Hillman:2002} it is shown that if $\pi = \cy 2\times \cy 2$ and  $\chi=1$
then the action of  $\pi$ on $\pi_2\cong\bZ^2$ is essentially unique, 
and in \cite[Chapter 12, \S6]{Hillman:2002} it is shown that just 3 of the elements of 
$H^3(\pi;\pi_2)=(\cy 2)^4$ are $k$-invariants of such $PD_4$-complexes .
Two of these $k$-invariants are interchanged by the involution which
swaps the orientation-reversing elements of $\pi$ and the summands of $\pi_2$ 
fixed by each such element.
Hence
there are exactly two equivalence classes of quadratic 2-types realized by 
$PD_4$-complexes $X$
with universal cover $\wX\simeq{S^2\times{S^2}}$ and  
$$\pi_1(X)\cong\pi=\langle{t,u}\mid{t^2=u^2=(tu)^2=1}\rangle.$$
Let $\{t^*,u^*\}$ be the dual basis for $H^1(\pi;\mathbb{F}_2)$. 
 If $X$ is a $PD_4$-complex with $\pi_1(X)=\pi$ and
$\chi(X)=1$,  then we may assume that $v_1(X)=t^*+u^*$ and 
$v_2(X)$ is either $t^*u^*$ or $t^*u^*+(u^*)^2$. 
This is an easy consequence of Poincar\'e duality with  coefficients $\mathbb{F}_2$ 
and the Wu formulas.

Let $X^+$ denote the orientation double cover of $X$.
If $v_2(X)=t^*u^*$ then $v_2(X^+)=t^{*2}\not=0$
and both non-orientable double covers have $v_2=0$, 
while if $v_2(X)=t^*u^*+(u^*)^2$ then $v_2(X^+)=0$
and just one of the non-orientable double covers has $v_2=0$.

The two possibilities for $v_2$ are realized respectively by $RP^2\times{RP^2}$ (with orientation double cover $S(3\eta)$) 
 and the nontrivial bundle space 
$RP^2\tilde\times{RP^2}={S^2\times{S^2}/\pi}$, where  $\pi$ acts by
$t(s,s')=(-s,s')$ and $u(s,s')=(R_\pi(s),-s')$, for all $s,s'\in{S^2}$.

 It now follows  that $N$ is in the quadratic $2$-type of $RP^2\tilde\times{RP^2}$, since its orientation double covering $N^{+} $ has $v_2(N^{+}) = 0$
 (see Remark \ref{rem:fivesix}). In particular, $N^{+} = S(\eta\oplus 2\epsilon)$.
\end{proof}

\begin{remark}
In Section \ref{sec:ten}, we will show that there are 
exactly four distinct homotopy types in the quadratic 2-type of $RP^2\tilde\times{RP^2}$, 
of which two are represented by manifolds. 
\end{remark}

\section{The quadratic 2-type of ${RP^2\times{RP^2}}$}\label{sec:nine}

In this section, we study the  quadratic 2-type for  the geometric quotient
${RP^2\times{RP^2}}$, and show that it contains only one homotopy type represented by a closed manifold.

Let $X_0=RP^2\times{RP^2}$, and let $t$ and $u$ be the generators of 
$\pi=\pi_1(X_0)$ corresponding to the factors.
Let $\Lambda=\Zpi $.
The inclusions of the factors $S^2$ into $\widetilde{X_0}$  
determine canonical generators $I_1$ and $I_2$ for $\Pi=\pi_2(X_0)$,
and  the $\Zpi$-module $\Pi$ is canonically split as  $\bZ_t\oplus\bZ_u$,
where $\bZ_t=\Lambda/(t+1,u-1)$
and $\bZ_u=\Lambda/(t-1,u+1)$.
Note that $\Hom_\Lambda(\bZ_t,\bZ_u)=\Hom_\Lambda(\bZ_u,\bZ_t)=0$.

Let $B$ be the Postnikov 2-stage of $X_0$, and let $w$, $\lambda$ be as defined in \S3.
If $f\in{ \Aut(B,w,\lambda)}$ let $f_1$ and $f_2$ be the induced automorphisms of $\pi$
and $\Pi$.
Then $$f_2(g\cdot\xi)=f_1(g)\cdot f_2(\xi)$$
 for all $g\in\pi$ and $\xi\in\pi_2$.
The isomorphism $f_1$ must  preserve the set of orientation reversing elements of $\pi$,
since $wf_1=w$.
Thus either $f_1=id_\pi$ or $f_1(t)=u$ and $f_1(u)=t$.
If $f_1=id_\pi$ then $f_2$ is $\Lambda$-linear, and 
so must respect the direct sum splitting of $\pi_2(X_0)$,
since $ \Hom_\Lambda(\bZ_t,\bZ_u)=
 \Hom_\Lambda(\bZ_u,\bZ_t)=0$.
Since $f_2$ must also be an isometry of the pairing $\lambda$
we see that $f_2=\pm{id_\Pi}$.
If $f_1\not=id_\pi$ then $f_2$ must transpose the generators of $\Pi$,
and again is determined up to sign.
Thus the image of $ \Aut(B,w,\lambda)$ in $ \Aut(\pi)\times{ \Aut(\Pi)}$ has  order  at most 4,
and so is abelian.

\begin{proposition}\label{prop:rp2xrp2}
 There are three homotopy types of $PD_4$-complexes $X_\alpha$ 
in the quadratic $2$-type of ${RP^2\times{RP^2}}$. 
\end{proposition}

\begin{proof}
Let $K=\overline{RP^2\times{RP^2}\setminus{D^4}}$,
and let $[J]$ be the homotopy class of a fixed lift $\widetilde{J}\colon S^3\to\widetilde{K}$
of the natural inclusion  $J\colon S^3=\partial{D}^4\to{K}$.
The Hurewicz homomorphism 
$h\colon \pi_3(K)\to{H_3(\widetilde{K};\bZ)}\cong\bZ^3$
is surjective, 
with kernel the image of $\Gamma_W(\Pi)$, 
generated by Whitehead products and composites with $\eta$.
Then $h([J])$ generates $H_3(\widetilde{K};\bZ)$ 
as a $\Lambda$-module, 
and $H_3(\widetilde{K};\bZ)\cong\Lambda/(1-t)(1-u)\Lambda$. 

The elements $\eta_1=I_1\circ\bar\eta$, $\eta_2=I_2\circ\bar\eta$ 
and $\zeta=[I_1,I_2]$
are a basis for $\Gamma_W(\Pi)\cong\bZ^3$.
Since $\Gamma_W(\Pi)$ is torsion free and $2\eta_i=[I_i,I_i]$,
we see that $t\eta_i=u\eta_i=\eta_i$ for $i=1,2$,
while $t\zeta=u\zeta=-\zeta$.
Hence $\bZ^w\otimes_\Lambda\Gamma_W(\Pi)\cong\bZ\oplus(\cy 2)^2$, 
and the torsion subgroup is generated by the images of $\eta_1$ and $\eta_2$.
In this case $B=P_2(X_0)$ is the product of two copies 
of the Postnikov 2-stage for $RP^2$,
and so the $k$-invariant is symmetric under the involution which interchanges the factors.
Hence it follows from Theorem \ref{thm:quad2} that there are at most
three homotopy types of $PD_4$-complexes $X_\alpha=K\cup_{[J]+\alpha}e^4$ 
in this quadratic 2-type, represented by $\alpha=0$, 
$\eta_1$ and $\eta_1+\eta_2$.

The transposition of the factors gives an element of $ \Aut(B,w,\lambda)$
which leaves the polarizations corresponding to 0 and 
$\eta_1+\eta_2$ invariant, 
while swapping the others.
 Since the image of $\Aut(B,w,\lambda)$ in $\Aut(\Pi)$ is abelian,
 this transposition is not conjugate in $ \Aut(B,w,\lambda)$ 
to an automorphism which fixes $\eta_1$ or $\eta_2$, 
we see that $X_{\eta_1+\eta_2}\not\simeq{X_{\eta_1}}$ or $X_{\eta_2}$.
It shall follow from Theorem 8.3 that $X_{\eta_1+\eta_2}$ 
is not homotopy equivalent to $X_0$,
and so the three homotopy types are distinct.
\end{proof}

\begin{remark} Let $\{t^*,u^*\}$ be the basis of $H^1(\pi;\mathbb{F}_2)$ dual to $\{t,u\}$.
Let $X_\alpha^t$ and $X_\alpha^u$ be the covering spaces associated to 
the subgroups $\langle{t}\rangle=\mathrm{Ker}(u^*)$ and
$\langle{u}\rangle=\mathrm{Ker}(t^*)$ of $\pi$, respectively.
It follows from Lemma \ref{lem:skeleton} that
since $(t^{*})^3=(u^{*})^3=0$ in $H^3(RP^2\times{RP^2};\mathbb{F}_2)$,
we have $(t^{*})^3=(u^{*})^3=0$ in $H^3(X_\alpha;\mathbb{F}_2)$, for all $\alpha$.
It follows easily from the nonsingularity of Poincar\'e duality that
the rings $H^*(X_\alpha;\mathbb{F}_2)$ are all isomorphic. 
In particular, $w_1(X_\alpha)=t^*+u^*$, $v_2(X_\alpha)=t^*u^*$
and $x^4=0$, for all $x\in{H^1(X_\alpha;\mathbb{F}_2)}$, in each case.
Hence $X_\alpha^+\simeq{S^2\times{S^2}/\langle\sigma^2\rangle}$,
while the non-orientable double covers $X_\alpha^t$ and $X_\alpha^u$ each have $v_2=0$. 
\end{remark}

We shall now adapt the argument of \cite[\S3]{Hambleton:1978} 
to show that if $\alpha\not=0$ then $X_\alpha$ 
is not homotopy equivalent to a closed 4-manifold.

\begin{theorem}\label{thm:four}
Let $M$ be a closed $4$-manifold with $\pi=\pi_1(M)=(\cy 2)^2$ and $\chi(M)=1$,
and such that $x^4=0$ for all $x\in{H^1(M;\mathbb{F}_2)}$.
Then $M$ is homotopy equivalent to $RP^2\times{RP^2}$.
\end{theorem}

\begin{proof}
Our hypotheses imply that $M$ is in the quadratic 2-type of $RP^2\times{RP^2}$,
and so $M\simeq{X_\alpha=K\cup_{[J]+\alpha}e^4}$, 
for some $\alpha=0$, $\eta_1$ or $\eta_1+\eta_2$.
For if $M$  were
in the quadratic 2-type of $RP^2\tilde\times{RP^2}$ 
then there  would be a class $x\in{H^1(M;\mathbb{F}_2)}$ such that $x^3\not=0$.
Poincar\'e duality considerations then imply that $x^4\not=0$ (see  
\cite[Chapter 12, \S\S4-6]{Hillman:2002}).

Suppose that $\alpha=\eta_1$ or $\eta_1+\eta_2$.
Then the image of $\alpha$ in $\pi_3(RP^2)$ under composition with 
the projection $pr_1$ to  the first factor is $\bar\eta$.
Hence the composite of the inclusion $K\subset{RP^2\times{RP^2}}$ with 
$pr_1$ extends to a map $p\colon X_\alpha\to{L=RP^2\cup_{\bar\eta}e^4}$
(note that $\mathrm{Ker}(\pi_1(p)=\langle{u}\rangle$). 
Let $\tilde{p}\colon X_\alpha^u\to\widetilde{L}$ be the induced map of double covers, 
and let $f\colon X_\alpha\to{RP^{k+1}}$ (for $k$ large) 
be the classifying map for the double cover $X_\alpha^u\to{X_\alpha}$.

Let $a=\tilde{p}^*(c)$ be the image of the generator of 
$H^2(\widetilde{L};\mathbb{F}_2)=\mathbb{F}_2$,
let $\bar{b}=(u^*)^2\in{H^2(X_\alpha;\mathbb{F}_2)}$, 
and let $b$ be the image of $\bar{b}$ in $H^2(X_\alpha^u;\mathbb{F}_2)$. 
The 3-skeleton of $X_\alpha^u$ is $K^u$, 
and so the covering transformation $t$ acts on $H^2(X_\alpha^u;\mathbb{F}_2)$ 
via the identity.
Hence the quadratic form $q$ defined in \cite[\S2]{Hambleton:1978},  and used in computing the Arf invariant
$A(X_\alpha,f)$ of the covering $X_\alpha^u\to{X_\alpha}$, 
is an enhancement of the ordinary cup product.

The pair $\{a,b\}$ is a symplectic basis with respect to the cup product,
and $q(a)=1$ since $\alpha \neq 0$,  by the argument of \cite[p.~1325]{Hambleton:1978}. 
Since $(u^{*})^3=(u^{*})^4=0$ in $H^*(X_\alpha;\mathbb{F}_2)$,
$Sq_i\bar{b}=Sq^{2-i}\bar{b}=0$ for $i=0$ or 1.
Hence  we also have $q(b)=1$, by \cite[Proposition 1.5]{Hambleton:1978},
and so $A(X_\alpha,f)$ is nonzero. But this contradicts the assumption that $X_\alpha$ is 
homotopy equivalent to a closed manifold,
by \cite[Proposition 2.2]{Hambleton:1978}, 
since any double covering of manifolds is Poincar\'e splittable. 
Hence $\alpha=0$ and so $M$ is homotopy equivalent to $RP^2\times{RP^2}$.
\end{proof}

\begin{corollary}
There is exactly one homotopy type for a closed manifold
in the quadratic $2$-type of  $RP^2\times{RP^2}$.
\end{corollary}

\begin{remark}
The inclusion $RP^2\to{L=RP^2\cup_{\bar\eta}e^4}$ induces isomorphisms on $\pi_i$ for $i\leq2$.
Since $L$ is covered by
$S^2\cup_\eta{e^4}\cup_{A\eta}e^4\simeq{S^2\cup_\eta{e^4}\vee{S^4}}=CP^2\vee{S^4}$,  
$\pi_3(L)=0$.
Hence we may view $L$  as the 4-skeleton of $P_2(RP^2)$.
(See \cite{Siegel:1967}.)
\end{remark}

\section{The quadratic 2-type of ${RP^2\tilde\times{RP^2}}$}\label{sec:ten}

In this section, we study the  quadratic 2-type for  the geometric quotient
$Y_0=RP^2\tilde\times{RP^2}$, 
and show that it contains exactly two homotopy types of closed 4-manifolds,  
represented by ${RP^2\tilde\times{RP^2}}$ and $N$ (see Definition \ref{def:N}). 
We shall need to examine the action of $ \Aut(B,w,\lambda)$
on $S^{PD}_4(B,w,\lambda)$ more closely than in the case of 
$RP^2\times{RP^2}$.

The manifold $Y_0$ is the total space of the non-trivial $RP^2$-bundle $p:Y_0\to{RP^2_b}$ 
 (where we add the subscript to distinguish the base, 
as the symbol $B$ is used elsewhere).
We may assume that $\pi=\pi_1(Y_0)$ is generated by $t$ and $u$,
where $t$ is in the image of the fibre $F\cong{RP^2}$ and 
 $u$ is the other orientation-reversing element.
Let $\{t^*,u^*\}$ be the Kronecker dual basis for $H^1(\pi;\mathbb{F}_2)$. 
Then $w_1(Y_0)=t^*+u^*$, and $u^{*3}=0$, 
since $u^*$ generates the image of $H^1(RP^2_b;\mathbb{F}_2)$.
The cohomology ring $H^*(Y_0;\mathbb{F}_2)$ is generated by 
$H^1(Y_0;\mathbb{F}_2)$, and $t^4\not=0$.
The latter two assertions follow from an elementary application of Poincar\'e duality
and the fact that $v_2(Y_0)\not=0$ (see \cite[Chapter 12,  \S4]{Hillman:2002}).

Let $\Lambda=\Zpi $.
The $\Lambda$-module $\pi_2(Y_0)$ is canonically split as 
$\bZ_t\oplus\bZ_u$,
where $\bZ_t=\Lambda/(t+1,u-1)$ is the image of $\pi_2(F)$
and $\bZ_u=\Lambda/(t-1,u+1)$ projects onto $\pi_2(RP^2_b)$.
Note that $ \Hom_\Lambda(\bZ_t,\bZ_u)=
 \Hom_\Lambda(\bZ_u,\bZ_t)=0$.

Let $B$ be the Postnikov 2-stage of $Y_0$, 
and let $w$, $\lambda$ be as defined in \S3.
If $f\in{ \Aut(B,w,\lambda)}$, let $f_1$ and $f_2$ be the induced automorphisms of $\pi$ and $\Pi$.
Then $f_2(g\cdot\xi)=f_1(g)\cdot f_2(\xi)$ for all $g\in\pi$ and $\xi\in\pi_2$.
The isomorphism $f_1$ must  preserve the set of orientation reversing elements of $\pi$,
since $wf_1=w$.
Thus either $f_1=id_\pi$ or $f_1(t)=u$ and $f_1(u)=t$.

Since we may construct $B$ by adding cells of dimension $\geq4$ to 
$Y_0$,  there is a homomorphism of truncated rings
$H^i(B;\mathbb{F}_2)\to{H^i(Y_0;\mathbb{F}_2)}$,
which is an isomorphism in degrees $i\leq2$ and a monomorphism in degree 3
(in fact an isomorphism in degree 3 also, 
since $H^*(Y_0;\mathbb{F}_2)$ is generated by $H^1(\pi;\mathbb{F}_2)$).
Since $u^{*3}=0$ and $t^{*3}\not=0$,  it follows that
$f_1$ cannot swap the generators $t$ and $u$.
Hence $f_1=id_\pi$, and so $f_2$ is $\Lambda$-linear.
Therefore $f_2$ must preserve each factor $Z_t$, $Z_u$ of the 
direct sum splitting of $\pi_2(Y_0)$, but possibly act as -1 on either summand, 
since $ \Hom_\Lambda(\bZ_t,\bZ_u)=
 \Hom_\Lambda(\bZ_u,\bZ_t)=0$.

\begin{proposition} 
\label{prop:rp2txrp2}
There are four homotopy types of $PD_4$-complexes $Y_\alpha$ 
in the quadratic $2$-type of ${RP^2\tilde\times{RP^2}}$. 
\end{proposition}

\begin{proof}

Let   $K'=\overline{RP^2\tilde\times{RP^2}\setminus{D^4}}$.
Let $J'\colon S^3=\partial{D}^4\to{K'}$ be the natural inclusion.
As outlined above,  $\Pi'=\pi_2(K')$ splits canonically as
$\Pi'=\bZ_t\oplus\bZ_u$.
Let $I_1'$ and $I_2'$ be  generators of $\bZ_t$ and $\bZ_u$,
respectively, 
and let $\eta_1'=I_1'\circ\eta$ and $\eta_2'=I_2'\circ\eta$ be the associated ``Hopf" maps.
Then $\{[I_1',I_2'],\eta_1',\eta_2'\}$ is a basis for
$\Gamma_W(\Pi')\cong\bZ^3$.
As in Propositions \ref{prop:s2xrp2} and \ref{prop:rp2xrp2},
we find that
$$\bZ^w\otimes_\Lambda\Gamma_W(\Pi')\cong\bZ\oplus(\cy 2)^2,$$
and the torsion subgroup is generated by the images of $\eta_1'$ and $\eta_2'$.
Thus there are at most four homotopy types of $PD_4$-complexes 
$Y_\alpha=K'\cup_{[J']+\alpha}e^4$ in this quadratic 2-type, 
represented by $\alpha=0$,  
$\eta_1'$, $\eta_2'$ and $\eta_1'+\eta_2'$,
by Theorem \ref{thm:quad2}.

 We now recall Whitehead's exact sequence  (see \cite[Theorem 2.3]{Kasprowski:2021b}):
$$ \dots \to H_4(\widetilde Y_\alpha) \to \Gamma_W(\pi_2(Y_\alpha)) \to \pi_3(Y_\alpha) \to 0$$
and the isomorphism $H_4(\wB;\bZ) \cong \Gamma_W(\pi_2(B))$.
The kernel of the ``Whitehead" homorphism from $\Gamma_W(\Pi')$ to
$\pi_3(Y_\alpha)$ is infinite cyclic, generated by the image of $J+\alpha$.
Since any map from $Y_\alpha$ to $Y_\beta$ covering an automorphism of $B$ 
must preserve the canonical basis for $\Pi'$ (up to signs), 
there can be no such map if $\alpha$ and $\beta$ are distinct elements of
the set $\{0,\eta_1',\eta_2', \eta_1'+\eta_2'\}$. 
Thus all four homotopy types are distinct. 
\end{proof}

\begin{remark} Let $\{t^*,u^*\}$ be the basis of $H^1(\pi;\mathbb{F}_2)$ dual to $\{t,u\}$,
and let $Y_\alpha^t$ and $Y_\alpha^u$ be the covering spaces associated to 
the subgroups $\langle{t}\rangle=\mathrm{Ker}(u^*)$ and
$\langle{u}\rangle=\mathrm{Ker}(t^*)$ of $\pi$, respectively.
We may assume that $u^*$ is induced from the base $RP^2$, so $(u^{*})^3=0$,
and then $(t^{*})^3\not=0$, since $v_2 (RP^2\tilde\times{RP^2})\not=0$.
It again follows from Lemma \ref{lem:skeleton} and Poincar\'e duality that
the $\mathbb{F}_2$-cohomology rings of the $Y_\alpha$ are all isomorphic. 
In particular, $v_2(Y_\alpha)=t^*u^*+(u^{*})^2$ in each case.
Hence in each case $Y_\alpha^+$ is homotopy equivalent to 
the $S^2$-bundle space over $RP^2$, which is a spin manifold.
The covering space $Y_\alpha^t$ is homotopy equivalent to $S^2\times{RP^2}$,
since $v_2(Y_\alpha^t)=0$,
while $Y_\alpha^u$ is homotopy equivalent to one of 
either $RP^4\#_{S^1}RP^4$ or $S^2\tilde\times{RP^2}$,
since $v_2(Y_\alpha^u)\not=0$.
\end{remark}

\begin{theorem}
Let $M$ be a closed $4$-manifold with $\pi=\pi_1(M)=(\cy 2)^2$ and $\chi(M)=1$,
and such that $x^4\not=0$ for some $x\in{H^1(M;\mathbb{F}_2)}$.
Then $M$ is homotopy equivalent to $RP^2\tilde\times{RP^2}$ or $N$.
\end{theorem}

\begin{proof}
We shall adapt the proof of Theorem \ref{thm:four}, again based on the arguments of \cite{Hambleton:1978}.
In this case $M$ must be in the quadratic 2-type of $RP^2\tilde\times{RP^2}$,
and so $M\simeq{Y_\alpha}=K'\cup_{[J']+\alpha}{e^4}$ for some  $\alpha=0$,  
$\eta_1'$, $\eta_2'$ or $\eta_1'+\eta_2'$.
The double covering space $M^t$ is homotopy equivalent to  $RP^2\times{S^2}$.
As in Theorem \ref{thm:four}, the covering automorphism induces the identity on $H^2(M^u;\mathbb{F}_2)$.

Suppose that $\alpha=\eta_2'$ or  $\eta_1'+\eta_2'$.
The composite of the inclusion $K'\subset{RP^2\tilde\times{RP^2}}$ with the bundle projection 
extends to a map $p\colon Y_\alpha\to{L}$.
Let $\tilde{p}\colon Y_\alpha^u\to\widetilde{L}$ be the induced map of double covers, 
and let
$a=\tilde{p}^*(c)$ be the image of the generator of $H^2(\widetilde{L};\mathbb{F}_2)$.
Let $\bar{b}=(t^{*})^2\in{H^2(Y_\alpha;\mathbb{F}_2)}$, 
and let $b$ be the image of $\bar{b}$ in $H^2(Y_\alpha^t;\mathbb{F}_2)$. 
Then $\{a,b\}$ is a symplectic basis for the cup product pairing.
We again find that $q(a)=q(b)=1$, so the Arf invariant associated to the 2-fold
covering $Y_\alpha^u\to{Y_\alpha}$ is nonzero,
contradicting the hypothesis that $M$ is a closed manifold.
Therefore either $\alpha=0$ or $\alpha=\eta_1'$.
Since  $Y_0=RP^2\tilde\times{RP^2}$ and $N$ are manifolds in this 
quadratic 2-type (see Proposition \ref{prop:sevenN}), 
and are not homotopy equivalent, we must have $Y_{\eta_1'}\simeq{N}$
and $M$ must be one of these two manifolds.
\end{proof}

The manifolds $RP^2\tilde\times{RP^2}$ and $N$ may be distinguished by
their (non-orientable) double covers.
However, 
we do not know whether $Y_{\eta_2'}\cong{Y_{\eta_1'+\eta_2'}}$.
Nor do we
Proposition \ref{prop:s2xrp2}
are double covers of the $PD_4$-complexes $X_\beta$ or $Y_\gamma$ of \S8 or \S9.

\section{Stable classification for  $\pi= \cy 4$}\label{sec:thirteen}
Let $\xi\colon B(w_1, w_2) \to{\BTop}$ denote the normal 1-type of the geometric quotient $M$ of $S^2 \times S^2$ 
with fundamental group $\pi = \cy 4$.
We may assume that $$B:= B(w_1, w_2) ={\BTopSpin}\times{K(\pi,1)},$$
since $w_2(\widetilde{M})=0$ 
(see \cite[Theorem 5.2.1 and \S8.1]{Teichner:1992}).
Let $c\colon M \to B$ denote the classifying map of the $\xi$-structure on $M$, and 
let $\gamma\colon B \to{K(\pi,1)}$ be the projection onto the second factor. 

We use a polarization $\gamma\circ c\colon M \to K(\pi,1)$  of $\pi_1(M)$ and fix a fundamental class  $[M] \in H_4(M;\bZ^w)$. This can be regarded as an ``orientation", since cap product with this class induces Poincar\'e duality for $M$ as a non-orientable manifold\footnote{Note Larry Taylor's remark ``non-orientable manifolds can not be oriented" \cite[\S 5]{Taylor:2008}.}{}.

The preferred local coefficient system $\{\bZ^w\}$
on $M$ pulled back from $K(\pi,1)$, followed by its pullback by $\gamma$,  gives a  preferred local coefficient system on $B$. 
Under the Thom isomorphism induced by the collapse map  $\varphi\colon S^{k+4}\to T(\nu_M)$,  for large $k$, the cap product $\varphi_*[S^{k+4}] \cap U(\nu_M)$ with a Thom class  gives a generator of $H_4(M;\bZ^w)$.  Hence a choice of fundamental class $[M] \in H_4(M;\bZ^w)$ determines  a preferred  generator  $U(\nu_M)\in H^k(T(\nu); \bZ^w) \cong \bZ$, and conversely (see \cite[\S 6]{Taylor:2008}).

Therefore, after fixing  a fundamental class for $M$, this construction provides  a preferred Thom class $U(\xi)$, and fixes a fundamental class  $[N] \in H_4(N;  \bZ^w)$  for each bordism element $[N,g] \in \Omega_4(B,\xi)$, by pull-back, since $g\colon N \to B$ is a lift of the classifying map $\nu_N\colon N \to \BTop$.

In order to compute the bordism group $\Omega_4(B,\xi)$ 
we use the Atiyah-Hirzebruch spectral sequence with
$E_{p,q}^2 = H_p(\pi; \Omega_q^{TOPSPIN})$
where the coefficients
\[
\Omega_q^{TOPSPIN} = \bZ, \cy 2, \cy 2, 0, \bZ,    \qquad \text{for\ } 0 \leq 1 \leq 4,
\]
are twisted by $w_1$ (and denoted $\bZ^w$).
We have $E_{p,0}^2 = H_p(\pi;\bZ^w) = \cy 2$, 
for $p$ even, and $E_{p,0}^2 = 0$ for $p$ odd. 
Similarly, $E_{0,4}^2 = \cy 2$.
The first differential
\[
d_2\colon E_{p,q}^2 \to E_{p-2,q+1}^2
\]
is dual to a map  on mod 2 cohomology 
\[
\hat d \colon H^{p-2}(\pi;\cy 2) \to H^p(\pi; \cy 2)
\]
for the cases $(4,2)$ and $(3,1)$. 

Note that the cohomology ring $H^*(\pi; \cy 2) = P(u) \otimes E(x)$, 
where $|u| = 2$ and $|x| = 1$, with $Sq^1 u = 0$ and $x^2 = 0$. 
The classes $w_1(\nu_M) = x$ and $w_2(\nu_M) = u$.

The  $d_2$ differentials starting
at $E_{*, 0}^2 = H_*(\pi; \bZ^w)$ factor through the reduction mod 2. 
According to Teichner \cite[\S 2]{Teichner:1997} 
the dual map $\hat d$ is given by the formula
\[
\hat d(\alpha) = Sq^2\alpha + (Sq^1\alpha)\cdot w_1 + \alpha\cdot  w_2.
\]
We compute using this formula and obtain:
\begin{align*}
\hat d \colon H^1(\pi;\cy 2) \to H^3(\pi;\cy 2),  \qquad&\hat d(x) = xu \neq 0\cr
\hat d \colon H^2(\pi;\cy 2) \to H^4(\pi;\cy 2),  \qquad&\hat d(u) =  0\cr
\hat d \colon H^3(\pi;\cy 2) \to H^5(\pi;\cy 2),  \qquad&\hat d(xu) = 0\cr
\hat d \colon H^4(\pi;\cy 2) \to H^6(\pi;\cy 2),  \qquad&\hat d(u^2) = u^3\neq0.
\end{align*}
After dualizing, we get $E_{0,4}^3 = \cy 2$, $E_{3,1}^3 = 0$, 
$E_{2,2}^3 = H_2(\pi;\cy 2) = \cy 2$, and $E_{4,0}^3 = \cy 2$.
Moreover, the only nonzero entry on the line $p+q=5$ of the $E^3$ page is $E^3_{3,2}=E^2_{3,2}=\bZ/2$. 

We remark that the non-zero element in $E_{0,4}^3 = \cy 2$  is represented by the image 
of the $E_8$-manifold under the inclusion map 
$$\Omega_4^{TOPSPIN}(\ast) \to \Omega_4(B, \xi).$$
However,  we have a factorization:
$$\Omega_4^{TOPSPIN}(\ast) \to \Omega_4(B, \xi) \to \Omega_4^{TOPSPIN^c}(\ast),$$
and the $E_8$-manifold represents a non-trivial element in 
$\Omega_4^{TOPSPIN^c}(\ast)$, as noted in \cite[p.~654]{Hambleton:1994}. Hence 
the $E_{0,4}^3$-term survives to $E_{0,4}^\infty$. 
The $E_{4,0}$-term is detected by the image of the twisted fundamental class.
Let $[N,g] \in \Omega_4(B,\xi)$ represent an element with $0 \neq\gamma_* g_*[N] \in H_4(\pi;\bZ^w)$. Then $N$ is non-orientable and $2[N,g] = 0$ from the null-bordism $g \circ p_1\colon N\times I \to B$. Hence there are no extensions in passing from $E_{*,*}^\infty$ to the bordism group. 
The conclusion is that 
\[
\Omega_4(B,\xi) = \cy 2 \oplus H_2(\pi;\cy 2) \oplus \cy 2. 
\] 
Recall that $c\colon M \to B$ denote the classifying map of the $\xi$-structure on $M$.
To detect elements in this bordism group, we can define
\[
\Omega_4(B,\xi)_M = \{ [M', c'] \colon  \gamma_*c'_*[M'] =  \gamma_*c_*[M] \in H_4(\pi;\bZ^w)\}.
\]
By Lemma \ref{lem:orderfour},  the image $\gamma_*c_*[M] \in H_4(\pi;\bZ^w)$ is non-zero. 
Therefore, $\Omega_4(B,\xi)_M$ is a coset of 
$$\ker(\gamma_*\colon \Omega_4(B,\xi) \to H_4(\pi; \bZ^w)) =  \cy 2 \oplus H_2(\pi;\cy 2). $$
Hence $\Omega_4(B,\xi)_M$ consists of four distinct bordism classes.

\medskip
Next, we introduce a related bordism theory. The pull-back diagram
\eqncount
\begin{equation}\label{eq:tenone}
\vcenter{\xymatrix{\BTopSpin  \ar[r] \ar@{=}[d]& \BMb  \ar[r]^j\ar[d]^\xi & M \ar[d]^{w_1(M) 
\times w_2(M)}\\
\BTopSpin \ar[r] & \BTop\ar[r] ^(0.3){w_1 \times w_2} & \ K(\cy 2, 1) \times K(\cy 2, 2).
}}
\end{equation}
defines a space $\BMb$ and a fibration $\xi\colon \BMb  \to \BTop$. We will now study the bordism groups  $\Omega_4(\BMb ,\xi)$ and the natural map 
 $$c_* \colon \Omega_4(\BMb ,\xi) \to \Omega_4(B(w_1, w_2),\xi).$$
See Kirby and Siebenmann \cite[p.~318]{Kirby:1977} for the low-dimensional homotopy groups of $\BTop$ and related spaces. In particular,
$\pi_4(\BTop) = \bZ \oplus \cy 2$ and the map 
$$\pi_4(\BTop) \to \pi_4(B(TOP/O)) = \pi_3(TOP/O) = \cy 2$$ is a split surjection. 
Topological bundles over $S^4$ are  classified by the stable triangulation class $k \in H^4(\BTop; \cy 2)$ and the first Pontrjagin class. 
Let $\zeta_0\colon S^4 \to \BTop$ be the topological bundle with $p_1(\zeta_0) = 0 $, and $k(\zeta_0) \neq 0$.

\begin{definition} We will define two reference maps for this bordism theory.  
\begin{enumerate}
\item
We can define $\widehat{\id} \colon M \to \BMb  $, since the map $(\id \times \nu_M) \colon M \to M \times \BTop$ factors through the pull-back
$ \BMb  \subset M \times \BTop $. 
\item Let
$\zeta_M := p^*(\zeta_0)$ be the pull-back  of the bundle $\zeta_0$  over the collapse
map $p\colon M \to S^4$.

 \item Similarly, we can define
$\widehat{\id}_{\zeta_M} \colon M \to \BMb \subset M \times \BTop$  by factoring $\id \times (\nu_M\oplus \zeta_M)$ through the pull-back \eqref{eq:tenone}.
\end{enumerate}
\end{definition}
Note that  $w_1(\zeta)= w_2(\zeta) = 0$, and the bundle  $\zeta_M\oplus \zeta_M$ is stably trivial.  By construction,  $\xi \circ \widehat{\id} = \nu_M$ and 
 $\xi \circ \widehat{\id}_\zeta = \nu_M\oplus \zeta$.

 \medskip
To shorten the notation, we will set $\BM: = \BMb$. A similar calculation to the one above shows that
$$\Omega_4(\BM,\xi) = \cy 2 \oplus H_2(M;\cy 2) \oplus H_3(M;\cy 2) \oplus H_4(M; \bZ^w)$$
with the filtration quotients
\begin{enumerate}
\item  $\cF_4/\cF_3 \cong E_{4,0}^\infty = H_4(M;\bZ^w)\cong \bZ$;
\item   $\cF_3/\cF_2 \cong E_{3,1}^\infty = H_3(M;\cy 2) = \cy 2$;
\item $\cF_2/\cF_0 \cong E_{2,2}^\infty = H_2(M;\cy 2)=\cy 2$;
\item $\cF_0 \cong E_{0,4}^\infty = H_0(M;\bZ^w) = \cy 2$.
\item $\cF_2 \cong H_2(M;\cy 2) \oplus \cy 2$, split by the KS-invariant.
\end{enumerate}

\medskip
 An element $[N, \hat g, \hat\nu]$ of this bordism group is represented by triple consisting of a closed $4$-manifold $N$ together with a reference map $\hat g\colon N \to \BM$, and a bundle map $\hat\nu \colon \nu_N \to \xi$  covering $\hat g$ (see Stong \cite[p.~14]{Stong:1968}, and Taylor \cite[\S 6]{Taylor:2008}). From the pull-back diagram \eqref{eq:tenone}, we have the composite $g:=j\circ \hat g\colon N \to M$.
  
 As above, the local coefficient system and choice of fundamental class for $N$ is determined by pull-back from $M$. 
 By composition with the classifying map $c\colon M \to B$, we obtain an element $c_*[N, \hat f, \hat\nu] \in \Omega_4(B, \xi)$. To simplify the notation, we will write $[N, \hat f]_\xi := [N, \hat f, \hat\nu]$. Since $B$ is the normal $1$-type of $M$, we have the structure
  $[M,  \widehat{\id}]_\xi \in \Omega_4(\BM,\xi)$ to serve as a base point.

\begin{lemma}\label{lem:tenone}  Let $M$ and $N$ be closed non-orientable $4$-manifolds  with universal covering $S^2 \times S^2$.  If $f\colon N \to M$ is a homotopy equivalence and $KS(M) = 0$, then  $f^*(\nu_M) \cong \nu_N$ if $KS(N) = 0$, and $f^*(\nu_M\oplus \zeta_M) \cong \nu_N$ if $KS(N) \neq 0$,
\end{lemma}
\begin{proof} 
It follows from the assumptions that  
$f^*(\nu_M)$ and $\nu_N$ have the same Stiefel-Whitney classes.
In particular, if $KS(N)=0$ then $f^*\nu_M-\nu_N$ lifts to an
orientable vector bundle $\lambda$ with $w_i(\lambda)=0$ for $i>0$.
By the Dold-Whitney classification \cite[Theorem 2(c)]{Dold:1959}, oriented vector bundles over a $4$-complex are stably determined by $p_1$ and $w_4$.
In our setting, the Pontrjagin class $p_1(\lambda)$ is divisible by 2, but
 $H^4(N;\bZ)=\cy 2$
since $N$ is non-orientable.
Hence $p_1(\lambda)=0$ and $\lambda$ is (stably) trivial. 
If $KS(N) \neq 0$,  then $f^*(\nu_M\oplus \zeta_M) - \nu_N$ lifts to an orientable vector bundle, which is again stably trivial.
\end{proof}

 Define the subset of degree one bordism elements:
 $$ \Omega_4( \BM, \xi)_M = \{ [N, \hat g]_\xi  :  g_*[N] = [M] \in H_4(M;\bZ^w)\}. $$
 A homotopy equivalence $f\colon N \to M$ represents an element of $ \cS(M)$. To define its normal invariant  
 $\eta(f) \in [M, G/TOP]$,  we can apply Lemma \ref{lem:tenone} to cover $f$ by a bundle map  to  $\nu_M$  or to  $(\nu_M\oplus \zeta)$,  if $KS(N) \neq 0$. 
A choice of bundle isomorphism  $f^*\nu_M \cong \nu_N$ (respectively, $f^*(\nu_M\oplus \zeta_M) \cong \nu_N$)  fixes a $\xi$-structure and a fundamental class for $N$ by pull-back,  so that $f\colon N \to M$  composed with $\widehat{\id}$ (respectively,  $\widehat{\id}_\zeta$)
represents an element $[N, \hat f]_\xi  \in  \Omega_4(\BM, \xi)$ with $\hat f \colon N \to \BM = \BMb$,  $f = j \circ \hat f$ and  $f_*[N] = [M]$.  

 \medskip
 The next result is an application of topological surgery (see \cite{Freedman:1982}).
 
 \begin{lemma} Every element in $ \Omega_4(\BM,\xi)_M$ has the form $[M' ,\hat f']_\xi$,  where $f'\colon M' \to M$ is  a homotopy equivalence. If $[M',\hat f']_\xi = [M'',\hat f'']_\xi$, where both $f'$ and $f''$ are homotopy equivalences, then there exists a homeomorphism $h \colon M' \to M''$ such that $f''\circ h \simeq f'$.
 \end{lemma}
 \begin{proof} Let $[N,\hat g]_\xi$ be an element in $ \Omega_4(\BM, \xi)_M$. Then $\hat g\colon N \to \BM$ together with its bundle data gives a
  $2$-connected map such that $g:=j\circ \hat g \colon N \to M$ has degree one. 
  Note that $K_2(\hat g) = K_2(j\circ \hat g)$ since $j$ is $3$-connected. Since $L_4(\cy 4, -) = 0$, modified surgery can be performed
   to obtain a homotopy equivalence $f'\colon M' \to M$ in the same $\xi$-bordism class. Here we are doing surgery on the map $\hat g\colon N \to \BM$ to eliminate the kernel group $K(\hat g) = K_2(j \circ \hat g) = \ker\{H_2(N;\La) \to H_2(M;\La)\}$ (compare \cite[\S 5]{Kreck:1999}).
 
If $[M',\hat f']_\xi = [M'',\hat f'']_\xi$, where both $f'$ and $f''$ are homotopy equivalences, then a $\xi$-bordism $(W,F)$ between these elements can be surgered (relative to the boundaries) to an $s$-cobordism since $L_5(\cy 4, -) = 0$. We then apply the topological $s$-cobordism theorem.
 \end{proof}
 
\begin{corollary}\label{cor:tenfive} The map $c_*\colon \Omega_4(\BM, \xi)_M \to \Omega_4(B, \xi)_M$ is surjective.  Every element in $ \Omega_4(B, \xi)_M$ has the form $c_*[M', \hat f']_\xi$,  where $f'\colon M' \to M$ is  a homotopy equivalence.
\end{corollary}
\begin{proof}  By comparing the spectral sequences, we see that the filtration subgroup $\cF_2 \subset \Omega_4(\BM,\xi)$ is mapped isomorphically into $\Omega_4(B, \xi)$. The term $E_{3.1}^\infty(M)$ is mapped  to zero and the term
$E_{4,0}^\infty(M) = \bZ$ is mapped surjectively onto $E_{4,0}^\infty(B) = \cy 2$.
\end{proof}

\begin{remark} Since $\Omega_4(\BM,\xi)_M$ has 8 elements, and both $\cS(M)$ and $\Omega_4(B, \xi)_M$ have 4 elements, the uniqueness statement for the representatives of $\Omega_4(\BM,\xi)_M$ implies that $M$ has some non-trivial self-homeomorphism. Indeed, the standard $\cy 4$-action  on $S^2 \times S^2$ generated by $\tau(s,t) = (-t,s)$ extends to a smooth  action of $D_8 = \la \tau, \sigma\ra$, where  $\sigma(s,t) = (s, -t)$.  Hence $\sigma$ induces an involution on $M$, which is not homotopic to the identity since $\sigma_*$ is non-trivial on homology. 
\end{remark}

The projection of the difference $[M', c \circ f] - [M,c]$ into $E_{2,2}^\infty(B) = H_2(\pi;\cy 2)$
is detected by the first component of  the normal invariant $\eta(f') \in [M,G/TOP]$,
with respect to the identification
\eqncount
\begin{equation}\label{eq:one}
\cS(M) = [M,G/TOP] \cong H^2(M;\cy 2) \oplus H^4(M;\bZ) \cong H_2(M;\cy 2) \oplus \cy 2.
 \end{equation}
given by Poincar\' e duality. 
We will call this the \emph{reduced normal invariant of $M$}, 
and denote by $\overline\eta(M')\in H_2(\pi;\cy 2)$ 
the equivalence class of $\eta(f')$ modulo the action 
on normal invariants by homotopy self equivalences of $M$. 
If this is zero, 
it follows that the difference $[M',c \circ f'] - [M,c]$ is detected by the KS invariant.

\begin{lemma}\label{lem:self}
Suppose that $f \colon M \to M$ is a self homotopy equivalence. 
Then the elements $(M, c\circ f)$ and $(M,c)$ are $\xi$-bordant. 
\end{lemma}

\begin{proof} By functoriality, the homotopy equivalence $f\colon M \to M$ induces 
a self homotopy equivalence $\phi\colon B \to B$, such that $c\circ f \simeq \phi\circ c$. 
However, since $B = \BTopSpin \times K(\pi,1)$ has the homotopy type of $K(\bZ, 4) \times K(\pi,1)$ 
through dimensions $\leq 5$, 
the composition $\phi\circ c$ is determined by the map $\phi^*\colon H^4(B;\bZ) \to H^4(B; \bZ)$. 
Either $\phi\circ c \simeq c$ or $\phi\circ c$ differs from $c$ by a non-trivial map $K(\pi,1) \to K(\bZ, 4)$. 
In the latter case,
the normal invariant of $f$ would have non-zero component in $H^2(\pi;\cy 2) \subset [M, G/TOP]$. 
But this would imply a change in the Kirby-Siebenmann invariant from domain to range of $f$,  by the formula in
\cite[p.~398]{Kirby:2001}, 
which is impossible for a self homotopy equivalence.
\end{proof}

\begin{corollary}
Stably homeomorphic manifolds homotopy equivalent to $M$ are homeomorphic. 
Such manifolds are distinguished by their reduced normal invariant and the KS invariant.
\end{corollary}

\begin{proof} According to the general theory of Kreck \cite{Kreck:1999}, to pass from bordism to the stable homeomorphism classification, we must consider the quotient of $\Omega_4(B, \xi)$ by the action of $\Aut(\xi)$. As pointed out by Kirby and Taylor \cite[pp.~394-395]{Kirby:2001}, it suffices to divide out the natural action of $\Out(\pi, w_1, w_2)$. 
The calculations above shows that this action is trivial, and hence  that the subset $\Omega_4(B,\xi)_M \subset \Omega_4(B,\xi)$   consists of 4 distinct stable homeomorphism classes, each represented by some homotopy equivalence $f\colon M' \to M$. 
However the structure set $S_{TOP}(M)$ has 4 elements (by Theorem \ref{thm:twoonet}), so there can be no non-trivial self homotopy equivalences. 
It follows that the choice of a homotopy equivalence $f\colon M' \to M$ is unique up to homotopy and composition with a homeomorphism. 
Hence the reduced normal invariant $\bar\eta(M') \in H_2(\pi;\cy 2)$ is a well-defined invariant of $M'$.
\end{proof}

\begin{proof}[The proof of Theorem A]
Here is a summary of the proof. Part (i) is proved in Lemma \ref{lem:orderfour}. By  Theorem \ref{thm:twoonet}, the structure set $S_{TOP}(M)$ has 4 elements, consisting of either two or four homeomorphism types of manifolds homotopy equivalent to $M$.  If $S_{TOP}(M)$ contained only two distinct homeomorphism types ($M$ and $\ast M$), then $M$ would admit a self-homotopy equivalence $(M,f)$ with non-trivial reduced normal invariant.  However, Lemma 10.8 shows that $(M,c\circ f)$ and $(M,c)$ are $\xi$-bordant. This would imply that the image of $S_{TOP}(M)$ in $ \Omega_4(B,\xi)$ would contain at most two distinct stable homeomorphism types. On the other hand,  Corollary   \ref{cor:tenfive} shows that $S_{TOP}(M)$ maps surjectively onto the subset  $\Omega_4(B,\xi)_M$, which consists of four distinct bordism classes  Hence no such self-equivalence of $M$ exists. This proves Parts (ii) and (iii) of Theorem A.
\end{proof}

\section{A  smooth fake version of $\bM$ ?}\label{sec:seven}

In this section we construct another smooth manifold $M''$ with $\pi_1(M'') = \cy 4$, which is homotopy equivalent to the geometric quotient $\bM$. At present we are not able to determine whether $M''$ is homeomorphic to $\bM$.

Let $M^+=S^2\times{S^2}/\langle\sigma^2\rangle=
S^2\times{S^2}/(s,s')\sim(A(s),A(s'))$ 
be the orientable double cover of $M=S^2\times{S^2}/\langle\sigma\rangle$.
Let $\Delta=\{(s,s)\mid{s}\in{S^2}\}$ be the diagonal in $S^2\times{S^2}$.
We may isotope $\Delta$ to a nearby sphere 
which meets $\Delta$ transversely in two points,
by rotating the first factor, 
and so $\Delta$ has self-intersection $\pm2$.
The diagonal is invariant under $\sigma^2$, 
and so $\delta=\Delta/\langle\sigma^2\rangle\cong{RP^2}$ embeds in $M^+$
with an orientable regular neighbourhood.
Since $\sigma(\Delta)\cap\Delta=\emptyset$ this also embeds in $M$.
We shall see that the complementary region also has a simple description.

We shall identify $S^3$ with the unit quaternions $\mathbb{H}_1$,
and view $S^2$ as the unit sphere in the space of purely imaginary quaternions.
The standard inner product on the latter space
is given by $v\bullet{w}=\mathfrak{Re}(v\bar{w})$,
for $v,w$ purely imaginary quaternions.
Let 
\[
C_x=\{(s,t)\in{S^2\times{S^2}}\mid{s\bullet{t}=x}\}, ~\forall~x\in[-1,1].
\]
Then $C_1=\Delta$ and $C_{-1}=\sigma(\Delta)$,
while $C_x\cong{C_0}$ for all $|x|<1$.
The map $f\colon S^3\to{C_0}$ given by $f(q)=(q\mathbf{i}q^{-1},q\mathbf{j}q^{-1})$
for all $q\in{S^3}$ is a 2-fold covering projection, and so $C_0\cong{RP^3}$.

It is easily seen that $N=\cup_{x\geq\varepsilon}C_x$ and $\sigma(N)$ 
are regular neighbourhoods of $\Delta$ and $\sigma(\Delta)$, 
respectively, while $C=\cup_{x\in[-\varepsilon,\varepsilon]}C_x\cong
{C_0\times[-\varepsilon,\varepsilon]}$.
In particular, $N$ and $\sigma(N)$ are each homeomorphic to the total space 
of the unit disc bundle in $T_{S^2}$, 
and $\partial{N}\cong{C_0}\cong{RP^3}$. 
The subsets $C_x$ are invariant under $\sigma^2$.
Hence $N(\delta)=N/\langle\sigma^2\rangle$ is the total space 
of the tangent disc bundle of $RP^2$.
In particular, $\partial{N(\delta)}\cong{L(4,1)}$ and
$\delta$ represents the nonzero element of $H_2(M;\mathbb{F}_2)$, 
since it has self-intersection 1 in $\mathbb{F}_2$.
\begin{remark}
It is not hard to show that any embedded surface representing
the nonzero element of $H_2(M;\mathbb{F}_2)$ is non-orientable
but lifts to $M^+$, and so has an orientable regular neighbourhood.
\end{remark}

We also see that $C/\langle\sigma^2\rangle\cong{L(4,1)}\times[-\varepsilon,\varepsilon]$.
Since $f(q.\frac1{\sqrt{2}}(\mathbf{1}+\mathbf{k}))=\sigma(f(q))$, 
the map $\tilde\sigma\colon S^3\to{S^3}$ defined by right multiplication
by $\frac1{\sqrt{2}}(\mathbf{1}+\mathbf{k})$ lifts $\sigma$.
Hence $C_0/\langle\sigma\rangle=S^3/\langle\tilde\sigma\rangle=L(8,1)$,
and so $MC=C/\langle\sigma\rangle$ is the mapping cylinder of 
the double cover $L(4,1)\to{L(8,1)}$.
Since $S^2\times{S^2}=N\cup{C}\cup\sigma(N)$ it follows that 
$M=N(\delta)\cup{MC}$.

This construction suggests a candidate for another smooth 4-manifold 
in the same (simple) homotopy type.
\begin{definition}\label{def:elevenone}
Let $M''=N(\delta)\cup{MC'}$, where $MC'$ is the mapping cylinder of 
the double cover $L(4,1)\to{L(8,5)}$.
Then $\pi_1(M'')\cong{\cy 4}$ and $\chi(M'')=1$, and so 
there is a homotopy equivalence $h\colon M''\simeq{M}$.
\end{definition}

Some questions for further investigation:
\begin{enumerate} 
\item Is there an easily analyzed explicit choice for $h\colon M'' \to M$,
with computable codimension two  Kervaire invariant?

\item Are $M$ and $M''$ homeomorphic? diffeomorphic?

\item Is there a computable homeomorphism (or diffeomorphism) invariant 
that can be applied here?
\end{enumerate}
We remark that most readily computable invariants are invariants of homotopy type.

%%%%%%%%%%%%%%%%%%
%\bibliographystyle{ih}
%\bibliography{s2s2}
%\end{document}
 %%%%%%%%%%%%%%%%%%%%%%%
\providecommand{\bysame}{\leavevmode\hbox to3em{\hrulefill}\thinspace}
\providecommand{\MR}{\relax\ifhmode\unskip\space\fi MR }
\providecommand{\MRhref}[2]{%
  \href{http://www.ams.org/mathscinet-getitem?mr=#1}{#2}
}
\providecommand{\href}[2]{#2}

\end{document}